\documentclass[11pt]{amsart}
\usepackage[margin=1in]{geometry}
\usepackage{amsmath}
\usepackage{amsfonts,amsmath}
\usepackage{paralist}
\usepackage[colorlinks=true]{hyperref}
\hypersetup{urlcolor=blue, citecolor=red}
\numberwithin{equation}{section}

\newtheorem{thm}{Theorem}[section]

\newtheorem{lem}[thm]{Lemma}

\newtheorem{defn}[thm]{Definition}

\newtheorem{clm}[thm]{Claim}

\usepackage{algorithmic}

\usepackage{graphicx,epstopdf}

\usepackage{xspace}
\usepackage{bold-extra}
\usepackage[most]{tcolorbox}

\colorlet{texcscolor}{blue!50!black}
\colorlet{texemcolor}{red!70!black}
\colorlet{texpreamble}{red!70!black}
\colorlet{codebackground}{black!25!white!25}

\lstdefinestyle{siamlatex}{%
	style=tcblatex,
	texcsstyle=*\color{texcscolor},
	texcsstyle=[2]\color{texemcolor},
	keywordstyle=[2]\color{texemcolor},
	moretexcs={cref,Cref,maketitle,mathcal,text,headers,email,url},
}

\tcbset{%
	colframe=black!75!white!75,
	coltitle=white,
	colback=codebackground, % bottom/left side
	colbacklower=white, % top/right side
	fonttitle=\bfseries,
	arc=0pt,outer arc=0pt,
	top=1pt,bottom=1pt,left=1mm,right=1mm,middle=1mm,boxsep=1mm,
	leftrule=0.3mm,rightrule=0.3mm,toprule=0.3mm,bottomrule=0.3mm,
	listing options={style=siamlatex}
}

\newtcblisting[use counter=example]{example}[2][]{%
	title={Example~\thetcbcounter: #2},#1}

\newtcbinputlisting[use counter=example]{\examplefile}[3][]{%
	title={Example~\thetcbcounter: #2},listing file={#3},#1}

\DeclareTotalTCBox{\code}{ v O{} }
{ %fontupper=\ttfamily\color{texemcolor},
	fontupper=\ttfamily\color{black},
	nobeforeafter,
	tcbox raise base,
	colback=codebackground,colframe=white,
	top=0pt,bottom=0pt,left=0mm,right=0mm,
	leftrule=0pt,rightrule=0pt,toprule=0mm,bottomrule=0mm,
	boxsep=0.5mm,
	#2}{#1}

\patchcmd\newpage{\vfil}{}{}{}
\newcommand{\RN}{\mathbb{R}^N}

\newcommand{\ot}{\Omega_T }

\newcommand{\po}{\partial\Omega}
\newcommand{\mdiv}{\textup{div}}

\newcommand{\io}{\int_{\Omega}}
\newcommand{\ioT}{\int_{\Omega_{T}}}

\newcommand{\ve}{\varepsilon}
\newcommand{\vp}{\varphi}
\newcommand{\vep}{\varepsilon}

\newcommand{\no}{n_1}
\newcommand{\nt}{n_2}

\newcommand{\uo}{n_1}
\newcommand{\ut}{n_2}

\newcommand{\vgm}{v^{(\gamma)}}
\newcommand{\we}{\nne}
\newcommand{\weg}{\nneg}
\newcommand{\nneg}{\left(n^{(\ep)}\right)}
\newcommand{\ome}{w^{(\ep)}}

\newcommand{\re}{R^{(\ep)}}

\newcommand{\nog}{n_1^{(\gamma)}}
\newcommand{\ntg}{n_2^{(\gamma)}}
\newcommand{\uoe}{n_1^{(\ep)}}
\newcommand{\ute}{n_2^{(\ep)}}
\newcommand{\de}{d^{(\ep)}}
\newcommand{\dg}{d^{(\gamma)}}
\newcommand{\nng}{n^{(\gamma)}}
\newcommand{\nnf}{n^{(\infty)}}

\newcommand{\wg}{w^{(\gamma)}}
\newcommand{\nne}{n^{(\ep)}}

\newcommand{\pt}{\partial_t}
\newcommand{\ep}{\varepsilon}

\newcommand{\ra}{\rightarrow}

\title[A tissue growth model with autophagy
] %Use the shortened version of the full title
{ Existence and incompressible limit of a tissue growth model with autophagy}

\author[Jian-Guo Liu and Xiangsheng Xu]{}

\keywords{	Autophagy; existence; incompressible limit; tissue growth models. {\it SIAM J. Math. Anal.}, to appear.%invariant sets, self-similar solutions,
}
\email{jliu@phy.duke.edu}
\email{xxu@math.msstate.edu}

\begin{document}
	\maketitle
	
	% Enter the first author's name and address:
	\centerline{\scshape Jian-Guo Liu and Xiangsheng Xu}
	\medskip
	{\footnotesize
		% please put the address of the first author
		\centerline{Department of Physics and Department of Mathematics} 
	\centerline{	Duke University}
	\centerline{ Durham, NC 27708, USA and}
		\centerline{Department of Mathematics \& Statistics}
		\centerline{Mississippi State University}
		\centerline{ Mississippi State, MS 39762, USA}
	} % Do not forget to end the {\footnotesize by the sign }
	
	\begin{abstract} In this paper we study a cross-diffusion system whose coefficient matrix is non-symmetric and degenerate.
		The system arises in the study of tissue growth with autophagy. The existence of a weak solution is established. We also investigate the limiting behavior of solutions as the pressure gets stiff. The so-called incompressible limit is a free boundary problem of Hele-Shaw type. Our key new discovery is that the usual energy estimate still holds as long as the time variable stays away from $0$.
		
	\end{abstract}
	
	\bigskip

%	-------------------------------------------

	\section{Introduction}
	%The problem
Let $\Omega$ be a bounded domain in $\mathbb{R}^N$ with Lipschitz boundary
$\po$ and $T$ any positive number. We consider the initial boundary value problem 
\begin{eqnarray}
\partial_t\no-\mdiv\left(\no\nabla p\right)&=& G(d)\no - K_1(d)\no + K_2(d)\nt\equiv R_1\nonumber\\
&& \mbox{in $\ot\equiv\Omega\times(0,T)$,}\label{euo}\\
\partial_t\nt-\mdiv\left(\nt\nabla p\right)&=& (G(d)-D)\nt +K_1(d)\no - K_2(d)\nt\equiv R_2\  \ \mbox{in $\ot$,}\label{eut}\\
b \pt d-\Delta d&=&-\psi(d)n+a\nt\ \ \mbox{in $\ot$,}\label{ed}\\
\no\nabla p\cdot\mathbf{n}&=&\nt\nabla p\cdot\mathbf{n}=0\  \ \mbox{on $\Sigma_T\equiv\po\times(0,T) $,}\label{uob}\\
%\nt\nabla p\cdot\mathbf{n}&=&0\  \ \mbox{on $\Sigma_T$,}\label{utb}\\
d&=&d_b\  \ \mbox{on $\Sigma_T$,}\label{cb}\\
%	\nabla c\cdot\mathbf{n}&=&0\  \ \mbox{on $\Sigma_T$,}\label{cb}\\
\hspace{.2in}	(\uo(x,0),\ut(x,0), d(x,0))&=&(n_{01}(x), n_{02}(x), d_0(x))\ \ \mbox{on $\Omega$, }\label{uicon}
\end{eqnarray}
where $\mathbf{n}$ is the unit outward normal to $\po$ and
%Let for some $w_p>0$
\begin{eqnarray}
n=\uo+\ut,\ \ p=n^\gamma,\ \ \gamma\geq 1.\label{wdef}
%	R_1&=&\no\fo(w)+\nt\go(w),\ \ \mbox{and}\label{r1}\\
%	R_2&=&\no\ft(w)+\nt\gt(w).\label{r2}
\end{eqnarray}
This problem was proposed as a tissue growth model with autophagy in \cite{DLZ}.  In the model, cells are classified into two phases: normal cells and autophagic cells, and
$n_1, n_2$ are their respective densities.  The third unknown function $d$ represents the concentration of nutrients. We assume that both cells have the same birth rate. Their death rates are different because autophagic cells have an extra death rate $D$ due
to the “self-eating” mechanism. % We use to describe . Then 
%it is reasonable to take 
Thus if  $G(d)$ is the net growth rate of normal cells then $G(d)-D$ gives the net growth rate for autophagic cells.
%, which
%	 is roughly the birth rate minus the death rate.  
Two types of cells can change from one to another.  The transition rates
are denoted by $K_1(d), K_2(d)$, respectively. Since autophagy is a reversible process, we have
\begin{equation}\label{con1}
K_1(d)\geq 0,\ \  K_2(d)\geq 0.
\end{equation}
Both cells consume nutrients with the consumption rate $\psi(d)$. However, autophagic cells also provide
nutrients by degrading its own constituents with a supply rate $a$. 
%This ODE model originally [17] considers
%cells cultured in a nutrient solution, and the nutrient solution is pumped in and out with the same rate $\lambda(t)$.
%	$\lambda(t)(c_B(t) - c)$ denotes the difference of nutrients between input flux and output flux.
%The key ingredient in this model is that while autophagic cells provide nutrients from the degradation of
%its own constituents with the nutrient supply rate $a$, the trade-off is an extra death rate $D$. 
We assume
\begin{equation}\label{con2}
D, a \in (0,\infty).
\end{equation}
Moreover,
\begin{equation}\label{dz}
	\psi(0) = 0, \ \mbox{$\psi(d)$ is increasing, and there is $d_0 > 0$ such that $ \psi(d_0) = a$}. 
\end{equation}
%	The net growth rate for normal cells $G(d)$ increases with the concentration of nutrients. The transition rates
%	$K_1(d),K_2(d)$ are positive since autophagy is a reversible process. When the nutrient is deficient, autophagy
%	is active and normal cells will change into autophagic cells faster. On the other hand,  when the nutrient is sufficient,
%	autophagic cells will turn back to normal cells faster. Thus it is natural for us to make 
%Mathematically, we have 
%	the following assumptions on
%	$G,K_1,K_2$:
%	\begin{equation}
%	G^\prime(d)\geq 0, K_1(d), K_2(d)>0,K_1^\prime(d)\leq 0, K_2^\prime(d) \geq 0.
%	\end{equation}
%As for $\psi (d)$ we assume it is non-negative and increasing with d. When 
The first condition in \eqref{dz} means that when there is no
nutrient the consumption rate should be zero. 
%Moreover, we assume there exists a 
The number $d_0$ is the so-called critical nutrient concentration. 
%such that $\psi(d_0) = a$.
% To summarize, we have the following assumptions on the consumption rate $\psi(d)$:
When $ d < d_0$
% and $\psi(d)-a $ are both negative,  
autophagic cells supplies more nutrients than they consume, while
%On the other hand, 
$d > d_0$ indicates that autophagic cells consumes more nutrients than they supply.

For the spatial motion of cells, we take a fluid mechanical point of view. That is, it is driven by a
velocity field equals to the negative gradient of the pressure (Darcy’s law) \cite{PQV}. And the pressure arises from
mechanical contact between cells. Denote by $p$ the pressure. Then we can assume that \eqref{wdef}, \eqref{euo}, and \eqref{eut} hold.
%Equations in \eqref{euo}-\eqref{wdef} are derived with the preceding consideration in mind, where
% \eqref{wdef} we arrive at .
% Here we assume $p$ is a power of total density $n = n_1 + n_2$, precisely:
%\begin{equation}\label{con3}
%p(x, t) =
%\frac{\gamma}{\gamma-1}
%n^{\gamma}(x, t), \gamma > 1.
%\end{equation}

One can also model tissue growth as free boundary problems \cite{F}. They are also called geometric or incompressible models and describe tissue as a moving
domain (see \cite{DP} and the references therein).
% where the density is constant. Free boundary problems arise also from the theory of mixture
%applied to tumor growth, [8, 9].
Building a link between these two classes of models has attracted the attention of many researchers
in recent years. The first result in this direction was obtained in \cite{PQV} for a purely mechanical model. It indicates that  the limit of the mechanical model  gives rise to a free boundary 
%by taking $\gamma\ra\infty$ in the mechanic model
%,one can obtain
%passing to 
problem
% are the so-called incompressible limit
 as the pressure becomes stiff. 
Since then the same result has been achieved
for a variety of models, which included active motion \cite{PQTV}, viscosity \cite{PV}, different laws of state \cite{DHV},
 more than one species of cells \cite{BPPS}, and multi-space dimensions and viscosity \cite{DPSV}. In each case the limit model turns out to be a free boundary
model of Hele-Shaw type.

The objective of this paper is to study the existence assertion for \eqref{euo}-\eqref{uicon} and the limiting behavior of solutions as $\gamma\ra\infty$.

We largely follow the approach adopted in \cite{PX} for the existence assertion. To understand the nature of the limiting model for our problem,
%The question of First let us take a look at this issue from the perspective of the nonlinear semigroup theory \cite{CF,X1}.
%our problem, 
%For this purpose, 
we define a family of maximal monotone graphs \cite{BCS} in $\mathbb{R}\times\mathbb{R}$ by
\begin{equation*}
	\vp_\gamma(s)=\left(s^+\right)^{\gamma+1}=\left\{\begin{array}{ll}
		s^{\gamma+1} &\mbox{if $s\geq 0$,}\\
		0&\mbox{if $s<0$}.
	\end{array}\right.
\end{equation*}
Obviously,
\begin{equation}\label{hope20}
	\vp_\gamma(s)\rightarrow \vp_\infty(s)\equiv\left\{\begin{array}{ll}
		[0,\infty) &\mbox{if $s= 1$,}\\
		0&\mbox{if $s<1$}
	\end{array}\right.
\end{equation}
in the sense of graphs as $\gamma\rightarrow\infty$ \cite{BCS}. The total density $n=\nng$ satisfies the problem
\begin{eqnarray}
	\pt\nng-\frac{\gamma}{\gamma+1}\Delta\vgm&=&G(d^{(\gamma)})\nng-Dn_2^{(\gamma)}\equiv R^{(\gamma)}\ \ \mbox{ in $\ot$},\nonumber\\
\vgm&=& \left(\nng\right)^{\gamma+1}\ \ \mbox{a.e. on $\ot$,}\label{hope12}\\
	\nabla \vgm\cdot\mathbf{n}&=&0\ \ \mbox{on $\Sigma_T$,}\nonumber\\
	\nng(x,0)&=&n_0\equiv n_{01}+n_{02}\ \ \mbox{on $\Omega$.}\nonumber
\end{eqnarray} 
Thus if we formally take $\gamma\rightarrow\infty$, we expect to arrive at the following problem
\begin{eqnarray}
	\pt n^{(\infty)}-\Delta v^{(\infty)}&=&G(d^{(\infty)})n^{(\infty)}-Dn_2^{(\infty)}\equiv R^{(\infty)}\ \ \mbox{ in $\ot$},\label{re4}\\
	v^{(\infty)}&\in&\vp_{\infty}(n^{(\infty)})\ \ \mbox{a.e. on $\ot$,}\label{re3}\\
	\nabla v^{(\infty)}\cdot\mathbf{n}&=&0\ \ \mbox{on $\Sigma_T$,}\label{re10}\\
	n^{(\infty)}(x,0)&=&n_0\ \ \mbox{on $\Omega$.}\label{re5}
\end{eqnarray} 
If $n_0\leq 1$ a.e on $\Omega$, a result of \cite{BCR} asserts that the limit problem \eqref{re4}-\eqref{re5} has an integral solution $n^{(\infty)}$ and $\lim_{\gamma\ra\infty}\nng= n^{(\infty)}$ in $L^1(0,T; L^1(\Omega))$ (also see \cite{X2} for related results). If $n_0> 1$ on a set of positive measure, the initial condition is no longer compatible with $\vp_{\infty}$ and the resulting problem \eqref{re4}-\eqref{re5} becomes singular.
%is no longer solvable via the generation theorem of nonlinear m-accretive operators \cite{BCR}. 
Thus identifying the limit of the sequence $\{\nng\}$ is an interesting issue.  When $R^{(\gamma)} \equiv 0$, this problem was solved in \cite{CF} through an application of the Aronson-B\'{e}nilan inequality \cite{AB}
\begin{equation}\label{abi}
	\pt\nng\geq -\frac{\nng}{\gamma t}.
\end{equation}
The precise result there is: If $\Omega=\RN$, $n_0(x)$ has a star-shaped profile, and $R^{(\gamma)}$=0,  then $ n^{(\infty)}\equiv\lim_{\gamma\ra\infty}\nng$ exists and is given by
%$$It turned out that the limit was the solution of 
$$n^{(\infty)}(x) =\left\{\begin{array}{ll}
	1& \mbox{if $x\in A$,}\\
	n_0(x)&\mbox{if $x\notin A$,}
\end{array}\right.$$
where $A$ is the coincident set of the solution of the following variational inequalities
$$-\Delta w\geq n_0-1,\ \ w\geq 0,\ \ \left(\Delta w+ n_0-1\right)w=0\ \ \mbox{in $\RN$.}$$
A remarkable fact is that the limit $ n^{(\infty)}$ is a function of $x$ only. A similar result was established for hyperbolic conservation laws in \cite{X2}. However, if  $R^{(\gamma)}$ changes sign, inequalities of the Aronson-B\'{e}nilan type no longer hold \cite{PTV}. To circumvent this difficulty, the authors of  \cite{DP}  established a weaker version of \eqref{abi} along with an $L^4$ estimate for the gradient of the pressure. Our problem here does not quite fit the framework developed in \cite{DP}. This forces us to take 
%This along with some other new estimates enabled them to pass to the limit in a problem that resembled ours. Here we take 
a totally different approach. It seems more convenient for us to work with $\vgm=\left(\nng\right)^{\gamma+1} $ instead of the pressure. Our key estimate is:
\begin{equation}
	\int_{\tau}^{T}\io\left(\vgm\right)^2dxdt+\int_{\tau}^{T}\io\left|\nabla\vgm\right|^2dxdt\leq \frac{c}{\tau} \ \ \mbox{for all $\gamma\geq 1$ and $\tau\in (0,T)$.}\nonumber
\end{equation}
Here and in what follows the letter $c$ denotes a generic positive constant whose value is determined by the given data. That is,  the sequence $\{\vgm\}$ is bounded in $L^2(\tau, T;W^{1,2}(\Omega))$ for each $\tau\in (0,T)$.
%We have managed to remove some technical assumptions in \cite{DP}. In particular, we do not require $G$ to be differentiable.
%, we need to have $n_0\leq 1$ for the problem to possess a solution. If this is not true, we have an incompatible problem. This is also called a singular perturbation.	

Before we introduce our remaining results, we state the definition of a weak solution.
\begin{defn}\label{defs}We say that $(\uo, \ut, d)$ is a weak solution to \eqref{euo}-\eqref{uicon} if:
	\begin{enumerate}
		\item[\textup{(D1)}]$\uo,\ut, d$ are all non-negative and bounded with
		\begin{equation}
		\partial_t\uo,\ \partial_t\ut,\ \pt d\in L^2(0,T; \left(W^{1,2}(\Omega)\right)^*),\  n^{\frac{\gamma+1}{2}}, \ d\in L^2(0,T; W^{1,2}(\Omega)),
		\end{equation}
		where $n$ is given as in \eqref{wdef} and $\left(W^{1,2}(\Omega)\right)^*$ denotes the dual space of $W^{1,2}(\Omega)$;
		\item[\textup{(D2)}]There hold
		\begin{eqnarray}
			\lefteqn{-\ioT\uo\partial_t\xi_1 dxdt+\ioT\uo\nabla n^\gamma\cdot\nabla\xi_1 dxdt}\nonumber\\
			&=&\ioT R_1\xi_1 dxdt-\langle \uo(\cdot, T), \xi_1(\cdot, T)\rangle+\io n_{01}(x)\xi_1(x, 0)dx\nonumber\\
			&& \mbox{for each $\xi_1\in H^1(0,T; W^{1,2}(\Omega))$, }\nonumber\\
			\lefteqn{	-\ioT\ut\partial_t\xi_2 dxdt+\ioT\ut\nabla n^\gamma\cdot\nabla\xi_2 dxdt}\nonumber\\
			&=&\ioT R_2\xi_2 dxdt-\langle \ut(\cdot, T), \xi_2(\cdot, T)\rangle+\io n_{02}(x)\xi_2(x, 0)dx\nonumber\\
			&& \mbox{for each $\xi_2\in H^1(0,T; W^{1,2}(\Omega))$, and}\nonumber\\
			\lefteqn{	-b\ioT d\partial_t\zeta dxdt+\ioT\nabla d\cdot\nabla\zeta dxdt}\nonumber\\
			&=&\ioT(-\psi(d)n+a\nt)\zeta dxdt-b\langle d(\cdot, T), \zeta(\cdot, T)\rangle+b\io d_0(x)\zeta(x, 0)dx\nonumber\\
			&& \mbox{for each $\zeta\in H^1(0,T; W_0^{1,2}(\Omega))$ },\nonumber
		\end{eqnarray}
		%for each smooth function $\varphi$, -\psi(d)n+a\nt
		where $\langle \cdot,\cdot\rangle$ denotes the duality pairing between $W^{1,2}(\Omega)$ and $\left(W^{1,2}(\Omega)\right)^*$ and $H^1(0,T; W^{1,2}(\Omega))=\{v\in L^2(0,T; W^{1,2}(\Omega)):\pt v\in L^2(0,T; W^{1,2}(\Omega))\}$;
		\item[\textup{(D3)}]  \eqref{cb} is satisfied.
	\end{enumerate}
\end{defn}
To see that the three equations in (D2) make sense, we can conclude from (D1) that
$\uo,\ut, d\in C([0, T]; \left(W^{1,2}(\Omega)\right)^*)$. Since $ n$ is bounded and
$\gamma\geq \frac{\gamma+1}{2}$, we also have $n^\gamma\in  L^2(0,T; W^{1,2}(\Omega))$.
\begin{thm}\label{mth} 
	%Let (H1)-(H3) be satisfied.  $G\geq 0$, 
	Assume:
	\begin{enumerate}
		\item[\textup{(H1)}] $G, K_1, K_2, \psi$ are all continuous functions;
		\item[\textup{(H2)}] \eqref{con1}, \eqref{con2}, and \eqref{dz} hold;
		\item[\textup{(H3)}]$b\in (0, \infty)$ and $\po$ is Lipschitz;
		\item[\textup{(H4)}] $n_{01}, n_{02}\in W^{1,2}(\Omega)\cap L^\infty(\Omega),\ d_0\in  L^\infty(\Omega),$ and $d_b\in  L^2(0,T; W^{1,2}(\Omega))\cap L^\infty(\ot)$. %\left(n_{01}+n_{02}\in\right)^{\gamma+1}\in W^{1,2}(\Omega)$;
		%\item[\textup{(H5)}] $\mu=\nu$.
	\end{enumerate}
	Then there is a weak solution to \eqref{euo}-\eqref{uicon}.	
\end{thm}% and 
%Let $d_0$ be given as in \eqref{dz}to \eqref{euo}-\eqref{uicon}
	Set
	\begin{eqnarray}
			L&=&\max\{\|d_b\|_{\infty,\Sigma_T},\|d_0\|_{\infty,\Omega}, d_0 \},\label{cl}\\
				G_0&=&\max_{s\in[0,L]} G(s).\label{gz}
	\end{eqnarray}

\begin{thm}\label{mth1}Let the assumptions of Theorem \ref{mth} hold. Assume:
	\begin{enumerate}
		\item[\textup{(H5)}] $G^\prime(s)$ is bounded;
		\item[\textup{(H6)}] $d_b\in W^{1,s}(\ot)$ for some $s>N+2$ and $d_0\in W^{1,\infty}(\Omega)$;
	\item[\textup{(H7)}] $\left|\left\{n_0(x)\geq \sigma\right\}\right|\leq \frac{1}{e^{G_0T}\|n_0\|_{\infty,\Omega}}|\Omega|$ for some $\sigma\in \left(0, e^{-G_0T}\right)$;
		\item[\textup{(H8)}] $\po$ is $C^{1,1}$.
	\end{enumerate}
Denote by $(\nng, \nog,\ntg,\dg)$ the solution obtained in Theorem \ref{mth}. 
Then there is a subsequence of $(\nng, \nog,\ntg,\dg)$, which will not be relabeled, such that
%as $\gamma\ra\infty$, we have
\begin{eqnarray}
(\nng, \nog,\ntg)&\ra&(n^{(\infty)},\no^{(\infty)},\nt^{(\infty)})\ \ \mbox{weak$^*$ in $\left(L^\infty(\ot)\right)^3$}\nonumber\\
&& \mbox{ and strongly in $\left(C([\tau,T]; \left(W^{1,2}(\Omega)\right)^*)\right)^3$ for each $\tau\in(0,T)$},\label{nwk}\\
	\vgm&\ra& v^{(\infty)}\ \ \mbox{weakly in $L^2(\tau, T;W^{1,2}(\Omega))$ for each $\tau\in (0,T)$},\label{vweak}\\
	\nabla\vgm&\ra&\nabla v^{(\infty)}\ \ \mbox{strongly in $L^2(\tau, T;(L^{2}(\Omega))^N)$ for each $\tau\in (0,T)$},\label{vstro}\\
%	\nog&\ra&n_1^{(\infty)}\ \ \mbox{weak$^*$ in $L^\infty(\ot)$},\label{now}\\
%			\ntg&\ra&n_2^{(\infty)}\ \ \mbox{weak$^*$ in $L^\infty(\ot)$},\label{ntw}\\
			\frac{	\ntg}{\nng}&\ra&\eta^{(\infty)}\ \ \mbox{weak$^*$ in $L^\infty(\ot)$},\label{ntw1}\\
				\dg&\ra&d^{(\infty)}\ \ \mbox{weak$^*$ in $L^\infty(0,T; W^{1,\infty}(\Omega))$ and strongly in $L^2(\ot)$}.\label{dstro}
				%\cap W^{1,2}(\ot)
\end{eqnarray}
The limit $(n^{(\infty)}, v^{(\infty)},n_1^{(\infty)},n_2^{(\infty)}, \eta^{(\infty)},d^{(\infty)} )$ satisfies
\begin{eqnarray}
	-\ioT n^{(\infty)}\pt\xi_1 dxdt+\ioT\nabla  v^{(\infty)}\cdot\nabla\xi_1 dxdt&=& \ioT R^{(\infty)}\xi_1 dxdt,\nonumber\\
	-\ioT n_1^{(\infty)}\partial_t\xi_2 dxdt+\ioT \left(1-\eta^{(\infty)}\right)\nabla  v^{(\infty)}\cdot\nabla\xi_2 dxdt
	&=&\ioT R_1^{(\infty)}\xi_2 dxdt,\nonumber\\
	-\ioT\ut^{(\infty)}\partial_t\xi_3 dxdt+\ioT \eta^{(\infty)} \nabla v^{(\infty)}\cdot\nabla\xi_3 dxdt
	&=&\ioT R_2^{(\infty)}\xi_3 dxdt,\ \ \mbox{and}\nonumber\\
	-b\ioT d^{(\infty)}\partial_t\xi_4 dxdt+\ioT\nabla d^{(\infty)}\cdot\nabla\xi_4 dxdt
	&=&\ioT(-\psi(d^{(\infty)})n^{(\infty)}+a\nt^{(\infty)})\xi_4 dxdt\nonumber\\
	&&-b\langle d^{(\infty)}(\cdot, T), \xi_4(\cdot, T)\rangle\nonumber\\
	&&+b\io d_0(x)\xi_4(x, 0)dx \nonumber
\end{eqnarray}
for each $(\xi_1,\xi_2,\xi_3)\in \left(H^1(0,T; W^{1,2}(\Omega))\right)^3$ with $(\xi_1,\xi_2,\xi_3)=0$ near $t=0$ and $\left.(\xi_1,\xi_2,\xi_3)\right|_{t=T}=0$ and each $\xi_4\in H^1(0,T; W^{1,2}_0(\Omega))$, where $R^{(\infty)}$ is given as in \eqref{re4} and
\begin{eqnarray}
	R_1^{(\infty)}&=&G(d^{(\infty)})\no^{(\infty)}-K_1(d^{(\infty)})\no^{(\infty)}+K_2(d^{(\infty)})\nt^{(\infty)},\nonumber\\
		R_2^{(\infty)}&=&\left(G(d^{(\infty)})-D\right)\nt^{(\infty)}+K_1(d^{(\infty)})\no^{(\infty)}-K_2(d^{(\infty)})\nt^{(\infty)}.\nonumber
\end{eqnarray}
%,}\ \ \mbox{for each $\vp\in H_0^1(0,T; W^{1,2}(\Omega))$, }\ \ \mbox{for each $\xi\in H^1_0(0,T; W^{1,2}(\Omega))$, and}\mbox{for each $\zeta\in H^1(0,T; W_0^{1,2}(\Omega))$ },
Moreover, \eqref{re3} holds and
\begin{equation}\label{com}
	v^{(\infty)}\left(\Delta 	v^{(\infty)}+R^{(\infty)} \right)=0.
\end{equation}
\end{thm}
%$	n^{(\infty)}$ satisfies \eqref{re4} and \eqref{re10} in the following sense 

If we compare the equations in (D2) with the ones here, two pieces are missing. One is that we are no longer able to identify the initial conditions for $(n^{(\infty)}, n_1^{(\infty)},n_2^{(\infty)})$. This is to be expected due to the fact that $\vp_{\infty}$ is not defined on the set $\{n_0>1\} $. A redeeming feature is that we can view  \eqref{com}, the so-called complementary condition, as some kind of compensation for this lack of initial conditions. More significantly, this condition connects our limits to
% is fundamental because it relates the weak solutions dened by
%the equations (4) and (5) to  and the closer $\|n_0\|_{\infty,\Omega}$ is to $1$
the geometric form of the Hele-Shaw problem \cite{DP}. At least formally, it says 
$$-\Delta 	v^{(\infty)}=R^{(\infty)}\ \ \mbox{ on $\Omega(t)\equiv\{v^{(\infty)}(x,t)>0\}$.}$$
The second one is that we have not been able to show
\begin{equation}\label{rev1}
	\eta^{(\infty)}=\frac{\nt^{(\infty)}}{n^{(\infty)}}.
\end{equation}
This can be derived from the precompactness of  $\{\nng\}$ in some $L^q(\ot)$ space with  $q\in [1, \infty)$ (see the proof of \eqref{ff} in Section 2 below). Unfortunately, this result is not available to us because in the generality considered here the sequence $\{\nabla\nng\}$ cannot be shown to be bounded in a function space. %is still an issue, which prevents us from obtaining \eqref{rev1} 
Furthermore, it does not seem to be possible to obtain any estimates on $\pt\vgm$ that are uniform in $\gamma$. As a result, the precompactness of  $\{\vgm\}$ in some $L^q(\ot)$ space is also an issue. This is so in spite of the fact that we have \eqref{vstro}.
%Fortunately, we are able to establish \eqref{vstro}, which is sufficient for our purpose.

 We can easily see that \eqref{re3} is equivalent to the following
\begin{eqnarray}
	n^{(\infty)}&\leq& 1 \ \ \mbox{on $\ot$ and}\label{sy1}\\
	\left(1-	n^{(\infty)}\right)	v^{(\infty)}&=& 0\ \ \mbox{ on $\ot$.}
	\label{sy4}
\end{eqnarray} Obviously, we can no longer expect $n^{(\infty)}$ to be independent of $t$ due to the presence of $R^{(\infty)}$.  The term $\Delta 	v^{(\infty)}$ may be a pure distribution. We define
$$	v^{(\infty)}\Delta 	v^{(\infty)}=\mdiv\left(	v^{(\infty)}\nabla 	v^{(\infty)}\right)-\left|\nabla v^{(\infty)}\right|^2\ \ \mbox{in the sense of distributions.}$$
 Also note that the assumption (H7) implies that $n_0$ is close to $0$ on a large set. The smaller $T$ is, the easier it is for (H7) to hold.

The remainder of the paper is devoted to the proof of the above two theorems. To be specific, Section 2 contains the proof of Theorem \ref{mth}, while Theorem \ref{mth1} is established in Section 3.
%\section{Preliminaries}
\section{Existence of a global weak solution and Proof of Theorem \ref{mth}}
The proof will be divided into several lemmas. Before we begin, we state the following three well known results.  
\begin{lem}\label{comt}Let $h(s)$ be a convex and lower semi-continuous function on $\mathbb{R}$ \cite{H}.
	% and $f$ a measurable function on $\ot$.
	% and $\alpha> 0$. 
	Assume that
	\begin{enumerate}
		\item[\textup{(C1)}] $f\in W_2(0,T)\equiv\left\{\vp\in L^2(0,T; W^{1,2}(\Omega)): \partial_t \vp\in L^2\left(0,T; \left(W^{1,2}(\Omega)\right)^*\right)\right\} $;
		%$\{f_n^\alpha\}$ is bounded in \partial_tL^2(0,T;\left(W^{1,2}(\Omega)\right)^*)Before we state our existence theorem, we let
		\item[\textup{(C2)}]
		%$\{f_n\}$ and $\{\partial_tf_n\}$ are bounded in
		$g\in L^2(0,T;W^{1,2}(\Omega))$ with the property $g(x,t)\in\partial h(f(x,t))$ for a.e $(x,t)\in \ot$, where $\partial h$ is the subgradient of $h$.
	\end{enumerate}
	Then 
	%$\{f_n\}$ is precompact in $C([0,T]; L^{\alpha+1}(\Omega))$. Furthermore, 
	the function $t\mapsto \io h(f(x,t))dx$ is absolutely continuous on $[0,T]$ and
	\begin{equation}\label{lm}
	\frac{d}{dt}\io h (f)dx=\left\langle\partial_tf,g\right\rangle.
	\end{equation}
	\end{lem}
If $h(s)=s^2$, this lemma is a special case of the well known Lions-Magenes lemma (\cite{T}, p.176–177).
Formula \eqref{lm} is trivial if $f$ is smooth. The general case can be established by suitable approximation. See (\cite{H}, p. 101) for the details.
%\begin{proof}Without loss of generality, assume that each function in the sequence $\{f_n\}$ is sufficiently smooth.  For $0\leq t_1<t_2\leq T$,
%	we have
%	\begin{eqnarray}
%	\io|f_n(x, t_2)-f_n(x, t_1)|^{\alpha+1}dx&\leq &\io\left|f_n^{\alpha+1}(x, t_2)-f_n^{\alpha+1}(x, t_1)\right|dx
%	\end{eqnarray}
\begin{lem}[Lions-Aubin]\label{la}
	Let $X_0, X$ and $X_1$ be three Banach spaces with $X_0 \subseteq X \subseteq X_1$. Suppose that $X_0$ is compactly embedded in $X$ and that $X$ is continuously embedded in $X_1$. For $1 \leq p, q \leq \infty$, let
	\begin{equation*}
	W_{p,q}(0,T)=\{u\in L^{p}([0,T];X_{0}): \partial_t u\in L^{q}([0,T];X_{1})\}.
	\end{equation*}
	%	{\displaystyle}
	Then:
	\begin{enumerate}
		\item[\textup{(i)}] If $p  < \infty$, then the embedding of $W_{p,q}(0,T)$ into $L^p([0, T]; X)$ is compact.
		\item[\textup{(ii)}] If $p  = \infty$ and $q  >  1$, then the embedding of $W_{p,q}(0,T)$ into $C([0, T]; X)$ is compact. 
	\end{enumerate}
\end{lem}
The proof of this lemma can be found in \cite{S}. We mention in passing that Lemmas \ref{comt} and \ref{la} imply
that $W_2(0,T)$ is contained in $C([0,T]; L^2(\Omega))$. 
\begin{lem}\label{poin}
	Let $\Omega$ be a bounded domain in $\RN$ with Lipschitz boundary and $1\leq p<N$.
	Then there is a positive number $c=c(N)$ such that
	\begin{equation}
		\|u-u_S\|_{p^*}\leq \frac{cd^{N+1-\frac{p}{N}}}{|S|^{\frac{1}{p}}}\|\nabla u\|_p \ \ \mbox{for each $u\in W^{1,p}(\Omega)$,}\nonumber
	\end{equation}
	where $S$ is any measurable subset of $\Omega$ with $|S|>0$, $u_S=\frac{1}{|S|}\int_S udx$, and $d$ is the diameter of $\Omega$.
\end{lem}	
This lemma can be inferred from Lemma 7.16 in \cite{GT}.

 Our approximate problems are similar to those in \cite{PX}. For each  $\varepsilon>0$, we consider
\begin{eqnarray}
\partial_tn-\varepsilon\Delta n&=&\gamma\mdiv\left( n^{\gamma}\nabla n\right)+G(d)\no+(G(d)-D)\nt\ \ \mbox{in $\ot$,}\label{aew}\\
\partial_t\uo-\varepsilon\Delta \uo&=&\gamma\mdiv\left(\uo n^{\gamma-1}\nabla n\right)+G(d)\no-K_1(d)\no\nonumber\\
&&+K_2(d)\nt\ \mbox{in $\ot$,}\label{aeuo}\\
\partial_t\ut-\varepsilon\Delta \ut&=&\gamma\mdiv\left(\ut  n^{\gamma-1}\nabla n\right)+(G(d)-D)\nt+K_1(d)\no\nonumber\\
&&-K_2(d)\nt \ \mbox{in $\ot$,}\label{aeut}\\
b \pt d-\Delta d&=& -\psi(d)n+a\nt\ \mbox{in $\ot$,}\label{aedt}\\
\nabla n\cdot\mathbf{n}=\nabla \uo\cdot\mathbf{n}&=&\nabla \ut\cdot\mathbf{n}=0\ \ \mbox{on $\Sigma_T$,}\label{autb}\\
d&=&d_b\ \ \mbox{on $\Sigma_T$,}\label{adb}\\
\left.\left(n,\uo,\ut,d\right)\right|_{t=0}&=&\left(n_0(x),n_{01}(x),n_{02}(x),d_0(x) \right)\ \mbox{on $\Omega$.}\label{auicon}
\end{eqnarray}

%That is, $\io u^2(x,t)dx\in C[0,T]$ if and only if $u\in C([0,T]; L^2(\Omega))$.
\begin{lem}\label{athm} Assume that \textup{(H1)}-\textup{(H4)} hold.
	%$G(d), K_1(d), K_2(d), \psi(d)$ are all continuous functions and \eqref{con1}, \eqref{con2}, \eqref{con3}, and \eqref{dz} are satisfied. 
	Then for each fixed $\ve>0$ there exists a quadruplet $(n, \uo, \ut, d)$ in the function space $\left(W_2(0,T)\right)^4\cap \left(L^\infty(\ot)\right)^4$ such that \eqref{aew}-\eqref{auicon} are all satisfied in the sense of Definition \ref{defs}.
%	\begin{enumerate}
%		\item[\textup{(1)}] $\uo\geq 0,\ut\geq 0, d\geq 0$, and $ n=\uo+\ut$;
%		\item[\textup{(2)}]	Equations  .
%	\end{enumerate}
\end{lem}
\begin{proof}This lemma
%Existence of a weak solution to \eqref{aew}-\eqref{auicon} 
will be established via the Leray-Schauder fixed point theorem (\cite{GT}, p.280). For this purpose,
we introduce a cut-off function
\begin{equation}\label{thdf}
\theta_\ell(s)=\left\{\begin{array}{cc}
0&\mbox{if $s\leq 0$,}\\
s&\mbox{if $0<s<\ell$,}\\
\ell&\mbox{if $s\geq \ell$,}
\end{array}\right.
\end{equation}
where $\ell>0$ will be selected as below.
%for each $\varepsilon\in (0,1)$.L^2\left(0,T;\left(W^{1,2}(\Omega)\right)^2\right)
We define an operator $\mathbb{M}$ from $\left(L^2(\ot)\right)^4$ into itself as follows: Let $( w,v_1,v_2,u)\in \left(L^2(\ot)\right)^4$. We first consider the initial boundary value problem
\begin{eqnarray}
\partial_tn-\mdiv\left[\ep+\gamma\left(\theta_\ell( v_1)+\theta_\ell( v_2)\right)\theta_\ell^{\gamma-1}(w)\nabla n\right]&=&\theta_\ell(v_1) G(\theta_\ell( u))\nonumber\\
&&+\left( G(\theta_\ell( u))-D\right)\theta_\ell(v_2)\ \mbox{in $\ot$,}\label{aew1}\\
\nabla n\cdot\mathbf{n}&=&0\ \ \mbox{on $\Sigma_T$,}\nonumber\\
n(x,0)&=&n_0(x)\ \ \mbox{on $\Omega$. }\label{aewi}
\end{eqnarray}
%Recall that $F,G$ are bounded functions. 
For given $( w,v_1,v_2, u)$ the above problem for $n$  is linear and uniformly parabolic. Thus we can conclude from the classical result (\cite{LSU}, Chap. III) that there is a unique weak solution $n$ to \eqref{aew1}-\eqref{aewi} in the space $W_2(0,T)$.
% owing to the well known Aubin–Lions lemma. compactly embedded
Use the function $n$ so obtained to form the following two initial boundary problems 
\begin{eqnarray}
\partial_t\uo-\varepsilon\Delta \uo&=&\gamma\mdiv\left[\theta_\ell(v_1)\theta_\ell^{\gamma-1}(w)\nabla n\right]+\left(G(\theta_\ell( u))-K_1(\theta_\ell( u))\right)\theta_\ell(v_1)\nonumber\\
&&+\theta_\ell(v_2)K_2(\theta_\ell( u))\ \mbox{in $\ot$,}\label{aeuo1}\\
\nabla \uo\cdot\mathbf{n}&=&0\ \ \mbox{on $\Sigma_T$,}\nonumber\\
\uo(x,0)&=&n_{01}(x) \ \mbox{on $\Omega$, }\nonumber\\
\partial_t\ut-\varepsilon\Delta \ut&=&\gamma\mdiv\left[\theta_\ell(v_2)\theta_\ell^{\gamma-1}(w)\nabla n\right]+\left(G(\theta_\ell( u))-K_2(\theta_\ell( u))-D\right)\theta_\ell(v_2)\nonumber\\
&&+\theta_\ell(v_1)K_1(\theta_\ell( u))\ \mbox{in $\ot$,}\label{aeut1}\\
\nabla \ut\cdot\mathbf{n}&=&0\ \ \mbox{on $\Sigma_T$,}\label{autb1}\\
\ut(x,0)&=&n_{02}(x) \ \ \mbox{on $\Omega$. }\label{auicon1 }
\end{eqnarray}
Each of the two problems here has a unique solution in $W_2(0,T)$. Then we solve the following linear problem
\begin{eqnarray}
b \pt d-\Delta d&=&-(\psi(\theta_\ell( u))-a)\theta_\ell( w)-a\theta_\ell(v_1)\ \mbox{in $\ot$,}\nonumber\\
d&=&d_b\ \ \mbox{on $\Sigma_T$,}\nonumber\\
d(x,0)&=&d_0(x) \ \ \mbox{on $\Omega$. }\nonumber
\end{eqnarray}
We define $(n,\uo,\ut, d)=\mathbb{M}( w,v_1,v_2, u)$.
% if $w,\uo,\ut$ are the respective weak solutions of the following initial boundary value problems
%Note that each equation in the preceding problems is linear and uniformly parabolic. Thus 
Evidently, $\mathbb{M}$ is well-defined.
\begin{clm} For each fixed pair $\varepsilon>0$ and $\ell>0$,
	the operator $\mathbb{M}$  is continuous and its range is precompact.
	%of maps bounded sets into precompact ones.is compact, i.e., $\mathbb{M}$
\end{clm}
\begin{proof}The key observation here is that each initial boundary value problem in the definition of $\mathbb{M}$ is linear and uniformly parabolic. This together with (H1) implies that $\mathbb{M}$ is continuous. One can easily verify that the range of  $\mathbb{M}$ is bounded in $\left(W_2(0,T)\right)^4$, which is compactly embedded in $\left(L^2(\ot)\right)^4$. It is similar to the proof of Lemma 2.4 in \cite{PX}. We shall omit the details.
	\end{proof}
Now we are in a position to apply Corollary 11.2 in (\cite{GT}, p.280), thereby obtaining that $\mathbb{M}$ has a fixed point. That is, there is a $(n,\uo,\ut, d)$ in $\left(W_2(0,T)\right)^4$ such that 
\begin{eqnarray}
	\partial_tn-\ep\Delta n&=&\gamma\mdiv\left[(\theta_\ell( \no)+\theta_\ell( \nt))\theta_\ell^{\gamma-1}(n)\nabla n\right]+ \theta_\ell(n_1) G(\theta_\ell( d))\nonumber\\
	&&+ \left( G(\theta_\ell( d))-D\right)\theta_\ell(n_2)\ \mbox{in $\ot$,}\label{aew2}\\
	\nabla n\cdot\mathbf{n}&=&0\ \ \mbox{on $\Sigma_T$,}\nonumber \\
	n(x,0)&=&  n_0(x)\ \ \mbox{on $\Omega$, }\nonumber\\
	\partial_t\uo-\varepsilon\Delta \uo&=&\gamma\mdiv\left[\theta_\ell(n_1)\theta_\ell^{\gamma-1}(n)\nabla n\right]+ \left(G(\theta_\ell( d))-K_1(\theta_\ell( d))\right)\theta_\ell(n_1)\nonumber\\
	&&+ \theta_\ell(n_2)K_2(\theta_\ell( d))\ \mbox{in $\ot$,}\label{aeuo2}\\
	\nabla \uo\cdot\mathbf{n}&=&0\ \ \mbox{on $\Sigma_T$,}\nonumber\\
	\uo(x,0)&=& n_{01}(x) \ \mbox{on $\Omega$, }\nonumber \\
	\partial_t\ut-\varepsilon\Delta \ut&=&\gamma\mdiv\left[\theta_\ell(n_2)\theta_\ell^{\gamma-1}(n)\nabla n\right]+ \left(G(\theta_\ell( d))-K_2(\theta_\ell( d))-D\right)\theta_\ell(n_2)\nonumber\\
	&&+ \theta_\ell(n_1)K_1(\theta_\ell( d))\ \mbox{in $\ot$,}\label{aeut2}\\
	\nabla \ut\cdot\mathbf{n}&=&0\ \ \mbox{on $\Sigma_T$,}\label{autb2}\\
	\ut(x,0)&=& n_{02}(x) \ \ \mbox{on $\Omega$, }\label{auicon2 }\\
	b \pt d-\Delta d&=&- (\psi(\theta_\ell( d))-a)\theta_\ell( n)-a \theta_\ell(n_1)\ \mbox{in $\ot$,}\label{efd}\\
	d&=&  d_b\ \ \mbox{on $\Sigma_T$,}\nonumber \\
	d(x,0)&=&  d_0(x) \ \ \mbox{on $\Omega$. }\label{efd1}
\end{eqnarray}
%	From here on we assume
Now we pick 
	\begin{equation}\label{elf}
	\ell\geq L,
	\end{equation}
where $L$ is given as in \eqref{cl}.
	Note that
	\begin{equation}
	\theta_\ell( d)= \min\{d, \ell\}.\nonumber
	\end{equation}
	On account of \eqref{dz}, we have
	\begin{equation}
	(\psi(\theta_\ell( d))-a)(d-L)^+=(\psi(\theta_\ell( d))-\psi(d_0))(d-L)^+\geq 0\ \ \mbox{in $\ot$.}\nonumber
	\end{equation}
	With this in mind, we use $(d-L)^+$ as a test function in \eqref{efd} to derive
	\begin{eqnarray}
	\lefteqn{\frac{ b }{2}\frac{d}{dt}\io\left[(d-L)^+\right]^2dx+\io \left|\nabla (d-L)^+\right|^2dx}\nonumber\\
	&&=\io\left[- (\psi(\theta_\ell( d))-a)\theta_\ell( n)-a \theta_\ell(n_1)\right](d-L)^+dx\leq 0.\nonumber
	\end{eqnarray}
	Integrate to obtain
	\begin{equation}\label{db}
	d\leq L\ \ \mbox{in $\ot$.}
	\end{equation}
	Note that 
	\begin{equation*}
	\theta_\ell(n_1)=0\ \ \mbox{in $\{\no\leq 0\}$}.
	%\ \ \mbox{and}\ \ \theta_\ell(n_2) \go(w)\geq 0\ \ \mbox{because $w\leq w_p$ and \eqref{fgc3} holds.}
	\end{equation*}
	With this in mind, we use $\uo^-$ as a test function in \eqref{aeuo1} to derive
	\begin{eqnarray*}
		-\frac{1}{2}\frac{d}{dt}\io\left(\uo^-\right)^2dx-\varepsilon\io|\nabla\uo^-|^2dx= \io\theta_\ell(n_2) K_2(\theta_\ell( d))\uo^-dx\geq 0.
	\end{eqnarray*}
	Consequently,
	%from whence follows
	%	This together with  implies
	\begin{equation*}
	\uo\geq 0.
	\end{equation*}
	By the same token, 
	\begin{equation*}
	\ut\geq 0.
	\end{equation*}
	Use $d^-$ as a test function in \eqref{efd} to get
	\begin{eqnarray}
	\lefteqn{-\frac{ b }{2}\frac{d}{dt}\io \left(d^-\right)^2dx-\io\left|\nabla d^-\right|^2dx}\nonumber\\
	&=&\io\left[- (\psi(\theta_\ell( d))-a)\theta_\ell( n)-a \theta_\ell(n_1)\right]d^-dx\nonumber\\
	&=&a \io\left[\theta_\ell( n)-\theta_\ell(n_1)\right]d^-dx\geq 0.\nonumber
	\end{eqnarray}
	Here we have used the fact that $\psi(0)=0$. Integrate to obtain
	\begin{equation}
	d\geq 0\ \ \mbox{in $\ot$.}
	\end{equation}
	This together with \eqref{db} implies
	%Subsequently,
	\begin{equation}
	\theta_\ell( d)=d.
	\end{equation}
	Add \eqref{aeuo2} to \eqref{aeut2} and subtract the resulting equation from \eqref{aew2} to derive
	\begin{equation*}
	\partial_t(n-(\uo+\ut))-\varepsilon\Delta(n-(\uo+\ut))=0\ \ \mbox{in $\ot$.}
	\end{equation*}
	Recall the initial boundary conditions for $(n-(\uo+\ut))$ to deduce
	\begin{equation}\label{wot}
	n=\uo+\ut.
	\end{equation}
	Let $\lambda\in (0,\infty)$, and define
	\begin{equation}
	w=e^{-\lambda t}n.
	\end{equation}
	We easily check that $w$ satisfies
	\begin{eqnarray}
	\partial_tw+\lambda w-\ep\Delta w&=&\gamma\mdiv\left[(\theta_\ell( \no)+\theta_\ell( \nt))\theta_\ell^{\gamma-1}(e^{\lambda t}w)\nabla w\right]+  e^{-\lambda t}\theta_\ell(n_1) G( d)\nonumber\\
	&&+  e^{-\lambda t}\left( G( d)-D\right)\theta_\ell(n_2)\ \mbox{in $\ot$,}\label{aew11}\\
	\nabla w\cdot\mathbf{n}&=&0\ \ \mbox{on $\Sigma_T$,}\nonumber \\
	w(x,0)&=&  n_0(x)\ \ \mbox{on $\Omega$. }\nonumber
	\end{eqnarray}
	%We estimate
	Set
	\begin{equation}\label{mz}
	M_0=\max\{ \max_{d\in [0,L]}|G(d)|, \max_{d\in [0,L]}|G(d)-D| \}.
	\end{equation}
	%Note that
	Then the last two terms in \eqref{aew11} can be estimated as follows:
	\begin{eqnarray}
	\lefteqn{\left|  e^{-\lambda t}\theta_\ell(n_1) G( d)
		+  e^{-\lambda t}\left( G( d)-D\right)\theta_\ell(n_2)\right|}\nonumber\\
	&\leq &  e^{-\lambda t}\theta_\ell(n_1) |G( d)|
	+  e^{-\lambda t}\left| G( d)-D\right|\theta_\ell(n_2)\nonumber\\
	&\leq &M_0  e^{-\lambda t}(\theta_\ell(n_1)+\theta_\ell(n_2))\nonumber\\
	&\leq &2M_0  e^{-\lambda t}\theta_\ell(n)\leq 2M_0w.\nonumber
	\end{eqnarray}
	It immediately follows that
	\begin{equation}
	\partial_tw+(\lambda-2M_0) w-\mdiv\left[\ep+\gamma(\theta_\ell( \no)+\theta_\ell( \nt))\theta_\ell^{\gamma-1}(e^{\lambda t}w)\nabla w\right]\leq 0 \ \mbox{in $\ot$.}\nonumber
	\end{equation}
	Choose $\lambda= 2M_0$. Then use $(w-\|n_0\|_{\infty, \Omega})^+$ as a test function in the above differential inequality to derive
	\begin{equation}
	w\leq \|n_0\|_{\infty, \Omega}\ \mbox{a.e. in $\ot$.}\nonumber
	\end{equation}
	This immediately implies
	\begin{equation}\label{nma}
	n\leq e^{2M_0 T}\|n_0\|_{\infty, \Omega}\ \mbox{a.e. in $\ot$.}
	\end{equation} 
	Thus if, in addition to \eqref{elf}, we further require
	\begin{eqnarray}\label{eft}
		\ell\geq e^{2M_0 T}\|n_0\|_{\infty, \Omega},
	\end{eqnarray}  then
	%Similarly, by using $w^-$ as a test function in \eqrefThus we can replace $\theta_\ell(w)$ in the preceding equations by $w$.
	%	In view of \eqref{thdf}, \eqref{wot}, and \eqref{wub}, we have
	\begin{equation*}
	\theta_\ell(n)=n,\ \	\theta_\ell(n_1)=\uo,\ \ \theta_\ell(n_2)=\ut
	\end{equation*}
and problem \eqref{aew2}-\eqref{efd1} reduces to problem \eqref{aew}-\eqref{auicon}.
This completes the proof of Lemma \ref{athm}.
\end{proof}

Let $\ep\in(0,1)$. Replace $n_{01}(x)$ by $n_{01}(x)+\ep$ in \eqref{auicon} and denote the resulting solution to \eqref{aew}-\eqref{auicon} by $(\nne,\uoe,\ute,\de)$. That is,  we have
\begin{eqnarray}
\partial_t\nne-\varepsilon\Delta \nne&=&\gamma\mdiv\left[ \nneg^{\gamma}\nabla \nne\right]+G(\de)\uoe\nonumber\\
&&+(G(\de)-D)\ute\ \ \mbox{in $\ot$,}\label{eew}\\
\partial_t\uoe-\varepsilon\Delta \uoe&=&\gamma\mdiv\left[\uoe \nneg^{\gamma-1}\nabla \nne\right]\nonumber\\
&&+G(\de)\uoe-K_1(\de)\uoe+K_2(\de)\ute\ \mbox{in $\ot$,}\label{eeuo}\\
\partial_t\ute-\varepsilon\Delta \ute&=&\gamma\mdiv\left[\ute  \nneg^{\gamma-1}\nabla \nne\right]+(G(\de)-D)\ute\nonumber\\
&&+K_1(\de)\uoe-K_2(\de)\ute \ \mbox{in $\ot$,}\label{eeut}\\
b \pt \de-\Delta \de&=& -\psi(\de)\nne+a\ute\ \mbox{in $\ot$,}\label{eedt}\\
\nabla \nne\cdot\mathbf{n}&=&\nabla \uoe\cdot\mathbf{n}=\nabla \ute\cdot\mathbf{n}=0\ \ \mbox{on $\Sigma_T$,}\label{eutb}\\
\de&=&d_b\ \ \mbox{on $\Sigma_T$,}\label{edb}\\
\left.(\nne,\uoe,\ute,\de)\right|_{t=0}&=&(n_0(x)+\ep,n_{01}(x)+\ep,n_{02}(x),d_0(x) )\ \mbox{on $\Omega$.}\label{euicon}
\end{eqnarray}
In addition, we have
\begin{eqnarray}
\uoe\geq 0,\ \ \ute&\geq& 0,\ \ \nne=\uoe+\ute\leq c,\nonumber \\
0&\leq&\de\leq c.\label{web}
\end{eqnarray}
Here and in what follows the letter $c$ is independent of $\ve$. As we shall see, the addition of $\ve$ in \eqref{euicon} is to ensure that $\nne$ stays away from $0$ below.
\begin{lem}\label{l41} %If $G\geq 0$, then
	We have
	\begin{eqnarray}
		\ioT\left|\nabla\left(\we\right)^{\frac{\gamma+1}{2}}\right|^2dxdt+\varepsilon\ioT\left(\left|\nabla\sqrt{\uoe}\right|^2+\left|\nabla\sqrt{\ute}\right|^2\right)dxdt\leq c.\nonumber
	\end{eqnarray}
\end{lem}
\begin{proof}
	Pick $\tau>0$. Use $\ln(\uoe+\tau)$ as a test function in \eqref{eeuo} to derive
	\begin{eqnarray}
	\lefteqn{\frac{d}{dt}\io\left((\uoe+\tau)\ln(\uoe+\tau)-\uoe\right)dx+\io \frac{\uoe}{\uoe+\tau}\nabla  \weg^\gamma\nabla\uoe dx}\nonumber\\
	&&+\varepsilon\io\frac{1}{\uoe+\tau}|\nabla\uoe|^2\nonumber\\
	&=&\io\left(G(\de)\uoe-K_1(\de)\uoe+K_2(\de)\ute \right)\ln(\uoe+\tau)dx\nonumber\\
	&\leq&\io\left|\left(G(\de)-K_1(\de)\right)\uoe \ln(\uoe+\tau)\right|dx+\int_{\{\uoe+\tau\geq 1\}}K_2(\de)\ute\ln(\uoe+\tau)dx\nonumber\\
%	&&-\int_{\{\uoe+\tau\leq 1\}}K_1(\de)\uoe\ln(\uoe+\tau)dx\nonumber\\
	&\leq &3C_0\io\we(\uoe+\tau)dx+2C_0\int_{\{\uoe+\tau\leq 1\}}|\uoe\ln\uoe| dx\leq c.\nonumber
	\end{eqnarray}
	Here 
	\begin{equation}\label{cz}
	C_0=\max\{\max_{d\in [0,L]}|G(d)|, \max_{d\in [0,L]}K_1(d),\max_{d\in [0,L]}K_2(d)\}.
	\end{equation}
	%, note that $\sup_{\ot}|\uoe\ln\uoe|\leq c$,
	Integrate and take $\tau\rightarrow 0$ to get
	\begin{equation*}
	%	\frac{1}{\mu}\io\left((\uoe+\tau)\ln\uoe-\uoe\right)dx
	\ioT\nabla\weg^{\gamma}\cdot\nabla\uoe dxdt+4\varepsilon\ioT\left|\nabla\sqrt{\uoe}\right|^2dxdt\leq c.
	\end{equation*}
	Similarly,
	\begin{equation*}
	\ioT\nabla\weg^{\gamma}\cdot\nabla\ute dxdt+4\varepsilon\ioT\left|\nabla\sqrt{\ute}\right|^2dxdt\leq c.
	\end{equation*}
	Add up the two preceding inequalities to obtain the desired result.
\end{proof}
\begin{lem}\label{l322} The sequences $\{\we\}$ and $ \{\de\}$ are precompact in $L^p(\ot)$ for each $p\geq 1$.
\end{lem}
\begin{proof}It follows from \eqref{mz} and \eqref{eew} that
	\begin{equation}
	\partial_t\nne-\varepsilon\Delta \nne\geq\gamma\mdiv\left[ \nneg^{\gamma}\nabla \nne\right]-M_0\nne\ \ \mbox{in $\ot$.}
	\end{equation}
	Let $w^{(\ep)}=e^{M_0t}\nne$. Then we have
	\begin{equation}\label{aew4}
	\partial_tw^{(\ep)}-\varepsilon\Delta w^{(\ep)}\geq\gamma\mdiv\left[ \nneg^{\gamma}\nabla w^{(\ep)}\right]\ \ \mbox{in $\ot$.}
	\end{equation}
	%	First, we assume
	%	\begin{equation*}
	%	\frac{\gamma+1}{2}>2.
	%	\end{equation*} 
	%	Without loss of generality, we may assume
	%This can be achieved easily by replacing $n_{01}$ with $n_{01}+\varepsilon$ in our approximate problems. From here on, we assume that we have already done so. 
	%Indeed,
	%Keep this and (H1) in mind and 
	Use $(\varepsilon-w^{(\ep)})^+$ as a test function in \eqref{aew4} to get
	\begin{equation}
	-\frac{1}{2}\frac{d}{dt}\io\left[(\varepsilon-w^{(\ep)})^+\right]^2dx-\gamma\io\weg^{\gamma}|\nabla(\varepsilon-w^{(\ep)})^+|^2dx-\varepsilon\io|\nabla(\varepsilon-w^{(\ep)})^+|^2dx
	\geq 0.
	\end{equation}
	Recall from \eqref{euicon} that $w^{(\ep)}(x,0)=	\we(x,0)\geq\varepsilon$.
	Integrate to obtain
	\begin{equation}\label{wep}
	\we\geq\varepsilon e^{-M_0T}. 
	\end{equation}
	Consequently, $\weg^r\in L^2(0,T; W^{1,2}(\Omega))$ for each $r\in \mathbb{R}$.
	We derive from \eqref{eew} that
	\begin{eqnarray}
	\partial_t\left(\we\right)^{\frac{\gamma+1}{2}}&=&\frac{\gamma+1}{2}\left(\we\right)^{\frac{\gamma+1}{2}-1}\partial_t\we\nonumber\\
	&=&\frac{\gamma+1}{2}\mdiv\left[\left(\we\right)^{\frac{\gamma+1}{2}}\nabla \weg^\gamma\right]-\frac{\gamma+1}{2}\nabla\left(\we\right)^{\frac{\gamma+1}{2}}\cdot\nabla \weg^\gamma\nonumber\\
	&&+\frac{(\gamma+1)\varepsilon}{2}\mdiv\left[\left(\we\right)^{\frac{\gamma+1}{2}-1}\nabla\we\right]-\frac{(\gamma+1)\varepsilon}{2}\nabla\left(\we\right)^{\frac{\gamma+1}{2}-1}\cdot\nabla \we\nonumber\\ &&+\frac{\gamma+1}{2}\left(\we\right)^{\frac{\gamma+1}{2}-1}(G(\de)\uoe+(G(\de)-D)\ute)\nonumber\\
	&=&\gamma\mdiv\left[\weg^{\gamma}\nabla\left(\we\right)^{\frac{\gamma+1}{2}} \right]-\frac{\gamma(\gamma-1)}{\gamma+1}\left(\we\right)^{\frac{\gamma+1}{2}-1}\left|\nabla\left(\we\right)^{\frac{\gamma+1}{2}}\right|^2\nonumber\\
	&&+\varepsilon\Delta\left(\we\right)^{\frac{\gamma+1}{2}}-(\gamma^2-1)\varepsilon\left(\we\right)^{\frac{\gamma+1}{2}-1}\left|\nabla \sqrt{\we}\right|^2\nonumber\\ &&+\frac{\gamma+1}{2}\left(\we\right)^{\frac{\gamma+1}{2}-1}\left(G(\de)\uoe+(G(\de)-D)\ute\right).\label{wcom}
	\end{eqnarray}
	Remember that $\frac{\gamma+1}{2}-1>0$. We can conclude from Lemma \ref{l41} that the sequence $\{\partial_t\left(\we\right)^{\frac{\gamma+1}{2}}\}$ is bounded in $L^2\left(0, T; \left(W^{1,2}(\Omega)\right)^*\right)+L^1(\ot)\equiv\{\psi_1+\psi_2:\psi_1\in L^2\left(0, T; \left(W^{1,2}(\Omega)\right)^*\right), \psi_2\in L^1(\ot) \}$. Now we are in a position to use (i) in Lemma \ref{la}, thereby obtaining the precompactness
	of $\{\left(\we\right)^{\frac{\gamma+1}{2}}\}$ in $L^2(\ot)$. 
	
	It is easy to see from \eqref{eedt} that $\{\de\}$ is bounded in $W_2(0,T)$.  The lemma follows from \eqref{web}.

\end{proof}

%This together with \eqref{aew4} implies that $\{\we\}$ is bounded in $W_2(0,T)$.weak$^*$ in $L^\infty(\ot)$, strongly in $L^2(\ot)$
We may extract a subsequence of $\{(\we, \uoe,\ute, \de)\}$, still denoted by the same notation, such that
\begin{eqnarray}
\we&\rightarrow&n \ \ \mbox{a.e. in $\ot$ and strongly in $L^p(\ot)$ for each $p\geq 1$,}\label{wc}\\
\de&\rightarrow&d \ \ \mbox{a.e. in $\ot$ and strongly in $L^p(\ot)$ for each $p\geq 1$,}\label{dc}\\
\uoe&\rightarrow&\uo \ \ \mbox{weak$^*$ in $L^\infty(\ot)$,}\nonumber \\
\ute&\rightarrow&\ut \ \ \mbox{weak$^*$ in $L^\infty(\ot)$, and}\nonumber \\
\left(\we\right)^{\frac{\gamma+1}{2}}&\rightarrow& n^{\frac{\gamma+1}{2}}\ \ \mbox{weakly in $L^2(0,T; W^{1,2}(\Omega))$ as $\ve\ra 0$.}
\end{eqnarray}
Since $\{\we\}$ is bounded, we also have
\begin{equation*}
\left(\we\right)^{p}\rightarrow n^{p}\ \ \mbox{weakly in $L^2(0,T; W^{1,2}(\Omega))$ for each $p\geq \frac{\gamma+1}{2}$.}
\end{equation*}
This combined with \eqref{aew4} implies
\begin{equation*}
\partial_t\we\rightarrow \partial_tn\ \ \mbox{weakly in $L^2(0,T; \left(W^{1,2}(\Omega)\right)^*)$.}
\end{equation*}
%The last limit implies
%In view of \eqref{ra1}, we infer
Remember that $G, K_1, K_2, \psi$ are all continuous functions. We also have
\begin{eqnarray}
G(\we)&\rightarrow &G(n)\ \ \mbox{strongly in $L^p(\ot)$ for each $p\geq 1$,}\label{r1c}\\ 
\psi(\we)&\rightarrow &\psi(n)\ \ \mbox{strongly in $L^p(\ot)$ for each $p\geq 1$,}\ \ \mbox{and }\label{r3c}\\ 
K_i(\we)&\rightarrow& K_i(n)\ \ \mbox{strongly in $L^p(\ot)$ for each $p\geq 1$, $i=1,2$.}\label{r2c}
\end{eqnarray}

%	\end{proof}
Our key result is the following.
\begin{lem}\label{l33}
Passing to a subsequence if necessary,	we have
	\begin{equation*}
	\left(\we\right)^{\gamma+1}\rightarrow n^{\gamma+1}\ \ \mbox{strongly in $L^2(0,T;W^{1,2}(\Omega))$.}
	\end{equation*}
\end{lem}
\begin{proof} We have
	\begin{equation}\label{ws1}
	\we \nabla(\we)^\gamma=\frac{\gamma}{\gamma+1}\nabla\left(\we\right)^{\gamma+1}.
	%\rightarrow \mu\gamma w^\gamma\nabla w\ \ \mbox{weakly in $\left(L^2(\ot)\right)^d$.}
	\end{equation}
	Thus we can write \eqref{eew} in the form
	\begin{equation}\label{aew44}
	\partial_t\we-\frac{\gamma}{\gamma+1}\Delta w^{(\varepsilon)}
	=\re,
	\end{equation}
	where
	\begin{eqnarray}
	w^{(\varepsilon)}&=&\left(\we\right)^{\gamma+1}+\frac{\varepsilon(\gamma+1)}{\gamma} \we,\nonumber \\
	\re&=&\left(G(\de)\uoe+(G(\de)-D)\ute\right).\nonumber
	\end{eqnarray}
	We may assume that $\we$ is a classical solution to \eqref{aew44} because it can be viewed as the limit of a sequence of classical approximate solutions. Use $\partial_t	 w^{(\varepsilon)}$ as a test function in \eqref{aew44} to derive
	\begin{equation}\label{end1}
	\io 	\partial_t\we\partial_t	 w^{(\varepsilon)} dx+\frac{\gamma}{\gamma+1}\io\nabla w^{(\varepsilon)}\cdot\nabla\partial_t	 w^{(\varepsilon)} dx=\io\re\partial_t	 w^{(\varepsilon)} dx
	\end{equation}
	We proceed to evaluate each integral in the above equation as follows:
	\begin{eqnarray*}
		\io 	\partial_t\we\partial_t	\ome dx&=&(\gamma+1)\io \left(\we\right)^{\gamma}\left(\partial_t\we\right)^2dx\nonumber\\
		&&+\frac{\varepsilon(\gamma+1)}{\gamma}\io\left(\partial_t\we\right)^2dx,\\
		\io\nabla w^{(\varepsilon)}\cdot\nabla\partial_t	\ome dx&=&\frac{1}{2}\frac{d}{dt}\io\left|\nabla\ome\right|^2dx,\\
		\io\re\partial_t	\ome dx&=&(\gamma+1)\io\re	\left(\we\right)^{\gamma}\partial_t\we dx\nonumber\\
		&&+\frac{\varepsilon(\gamma+1)}{\gamma}\io\re\partial_t\we dx\nonumber\\
		&\leq &\frac{\gamma+1}{2}\io \left(\we\right)^{\gamma}\left(\partial_t\we\right)^2dx\nonumber\\
		&&+\frac{\gamma+1}{2}\io \left(\we\right)^{\gamma}\left(\re\right)^2dx\nonumber\\
		&&+\frac{\varepsilon(\gamma+1)}{2\gamma}\io\left(\partial_t\we\right)^2dx+\frac{\varepsilon(\gamma+1)}{2\gamma}\io\left(\re\right)^2dx.
	\end{eqnarray*}
	Plug the preceding three results into \eqref{end1} and integrate to derive
	\begin{equation*}
	\ioT\left(\partial_t\weg^{\frac{\gamma+2}{2}}\right)^{2}dxdt+\varepsilon\ioT\left(\partial_t\we\right)^2dxdt+\sup_{0\leq t\leq T}\io\left|\nabla\ome\right|^2dx\leq c.
	\end{equation*}
	%Remember $\gamma>1$. Consequently,
	Note
	\begin{eqnarray}
	\partial_t\left(\we\right)^{\gamma+1}&=&2\weg^{\frac{\gamma+2}{2}}\partial_t\weg^{\frac{\gamma+2}{2}},\nonumber \\
	\nabla \left(\we\right)^{\gamma+1}&=&(\gamma+1)\left(\we\right)^{\gamma}\nabla\we.\nonumber
	\end{eqnarray}
	%Since $\gamma+1>\frac{\gamma+2}{2}$, we also have that 
	On account of \eqref{web},	$\{\partial_t\left(\we\right)^{\gamma+1}\}$ is bounded in $L^2(\ot)$, while $\{\left(\we\right)^{\gamma+1}\}$ is bounded in $L^\infty(0,T; W^{1,2}(\Omega))$. By (ii) in Lemma \ref{la}, the sequence $\{\left(\we\right)^{\gamma+1}\}$ is precompact in $C([0, T], L^2(\Omega))$.
	%	It immediately follows from the boundedness of $\{\we\}$ that 
	Consequently, $\{\weg^{\gamma+1}\}$ is precompact in $C([0, T], L^p(\Omega))$ for each $p\geq 1$. This asserts
	\begin{equation}\label{pc}
	\io\left(\we(x,t)\right)^qdx\rightarrow\io n^q(x,t)dx\ 
	\ \mbox{for each $t\in [0,T]$ and each $q\geq \gamma+1$}
	\end{equation}
	(pass to a subsequence if need be.)
	
	Take $\varepsilon\rightarrow 0$ in \eqref{aew44} to obtain
	\begin{equation*}
	\partial_tn-\frac{\gamma}{\gamma+1}\Delta n^{\gamma+1}= R\equiv G(d)\no+(G(d)-D)\nt.
	\end{equation*}	
	Subtract this equation from \eqref{aew44} and keep \eqref{ws1} in mind to get
	\begin{eqnarray}
	\partial_t(\we-n)-\frac{\gamma}{\gamma+1}\Delta\left[\left(\we\right)^{\gamma+1}- n^{\gamma+1}\right]-\varepsilon\Delta \we
	&=&\re-R.\label{ws2}
	\end{eqnarray}
	%Remember
	%It is easy to see that
	%\begin{equation*}
	%\re\rightharpoonup R\ \ \mbox{weakly in $L^p(\ot)$ for each $p\geq 1$.}
	%\end{equation*}
	Use $\left(\we\right)^{\gamma+1}- n^{\gamma+1}$ as a test function in \eqref{ws2} to derive
	\begin{eqnarray}
	\lefteqn{\frac{\gamma}{\gamma+1}\ioT\left|\nabla\left[\left(\we\right)^{\gamma+1}- n^{\gamma+1}\right]\right|^2dxdt}\nonumber\\
	&&+\varepsilon\ioT\nabla\we\cdot\nabla\left[\left(\we\right)^{\gamma+1}- n^{\gamma+1}\right]dxdt\nonumber\\
	&=& \ioT (\re-R)\left[\left(\we\right)^{\gamma+1}- n^{\gamma+1}\right]dxdt\nonumber\\
	&&-\int_0^T\left\langle\partial_t(\we-n),\left(\we\right)^{\gamma+1}- n^{\gamma+1}\right\rangle dt.\label{ws4}
	\end{eqnarray}	
	We will show that the last three terms in the above equation all go to $0$ as $\varepsilon\rightarrow 0$.
	It is easy to see from Lemma \ref{l41} that
	\begin{eqnarray}
	\lefteqn{\left|\varepsilon\ioT\nabla\we\cdot\nabla\left[\left(\we\right)^{\gamma+1}- n^{\gamma+1}\right]dxdt\right|}\nonumber\\
	&=&4\varepsilon\left|\ioT\sqrt{\we}\nabla\sqrt{\we}\cdot\left[\left(\we\right)^{\frac{\gamma+1}{2}}\nabla\left(\we\right)^{\frac{\gamma+1}{2}}- n^{\frac{\gamma+1}{2}}\nabla n^{\frac{\gamma+1}{2}}\right]dxdt\right|\nonumber\\
	&\leq &c\sqrt{\varepsilon}\rightarrow 0\ \ \mbox{as $\varepsilon\rightarrow 0$.}\nonumber
	\end{eqnarray}
	Obviously, we have
	\begin{equation*}
	\ioT (\re-R)\left[\left(\we\right)^{\gamma+1}- n^{\gamma+1}\right]dxdt\rightarrow 0\ \ \mbox{as $\varepsilon\rightarrow 0$.}
	\end{equation*}
	Finally,  we compute from Lemma \ref{comt} and \eqref{pc} that
	\begin{eqnarray*}
		\lefteqn{\int_0^T\left\langle\partial_t(\we-n),\left(\we\right)^{\gamma+1}- n^{\gamma+1}\right\rangle dt}\nonumber\\
		&=&\frac{1}{\gamma+2}\int_{0}^{T}\left[\frac{d}{dt}\io\left(\we\right)^{\gamma+2}dx+\frac{d}{dt}\io n ^{\gamma+2}dx\right]dt\nonumber\\
		&&-\int_0^T\left\langle\partial_t\we, n^{\gamma+1}\right\rangle dt-\int_0^T\left\langle\partial_t n, \weg^{\gamma+1}\right\rangle dt\nonumber\\
		&=&\frac{1}{\gamma+2}\left[\io\left(\we(x,T)\right)^{\gamma+2}dx+\io n^{\gamma+2}(x,T) dx\right]\nonumber\\
		&&-\frac{2}{\gamma+2}\io\left(n_0(x)\right)^{\gamma+2}dx-\int_0^T\left\langle\partial_t\we, n^{\gamma+1}\right\rangle dt\nonumber\\
		&&-\int_0^T\left\langle\partial_t n, \weg^{\gamma+1}\right\rangle dt\nonumber\\
		&\rightarrow&\frac{2}{\gamma+2}\io n^{\gamma+2}(x,T) dx-\frac{2}{\gamma+2}\io\left(n_0(x)\right)^{\gamma+2}dx-2\int_0^T\left\langle\partial_tn, n^{\gamma+1}\right\rangle dt\nonumber\\
		&=&0.
	\end{eqnarray*}
	This completes the proof.
\end{proof}
\begin{proof}[Proof of Theorem \ref{mth} ]
	Equipped with the preceding lemmas, we can complete the proof of Theorem \ref{mth}.
	Keeping \eqref{wep} in mind, we can
	set
	\begin{equation*}
	\eta_1^{(\varepsilon)}=\frac{\uoe}{\we},\ \ \eta_2^{(\varepsilon)}=\frac{\ute}{\we}.
	\end{equation*}
	Suppose
	\begin{equation*}
	\eta_1^{(\varepsilon)}\rightarrow \eta_1,\ \ \eta_2^{(\varepsilon)}\rightarrow \eta_2\ \ \mbox{weak$^*$ in $L^\infty(\ot)$.}
	\end{equation*}
	We calculate
	\begin{eqnarray}
	\uoe\nabla\weg^{\gamma}&=&\eta_1^{(\varepsilon)}\we\nabla\weg^{\gamma}\nonumber\\
	&=&\frac{\gamma}{\gamma+1}\eta_1^{(\varepsilon)}\nabla\weg^{\gamma+1}\nonumber\\
	&\rightarrow &\frac{\gamma}{\gamma+1}\eta_1\nabla n^{\gamma+1}
	=\eta_1n\nabla n^{\gamma} \ \mbox{weakly in $\left(L^2(\ot)\right)^N$.}\nonumber
	\end{eqnarray}
	We claim that
	\begin{equation}\label{ff}
	\eta_1n=\uo\ \ \mbox{a.e. on $\ot$}.
	\end{equation}
	To see this, for each $\delta>0$ we deduce from Lemma \ref{l322} that
	\begin{equation*}
	\eta_1^{(\varepsilon)}(\we-\delta)^+\rightarrow \eta_1(n-\delta)^+\ \ \mbox{weak$^*$ in $L^\infty(\ot)$.}
	\end{equation*}
	Note that $\frac{(\we-\delta)^+}{\we}\leq 1$.
	%On the other hand,
	As a result, we have
	\begin{equation*}
	\eta_1^{(\varepsilon)}(\we-\delta)^+=\uoe\frac{(\we-\delta)^+}{\we}\rightarrow \uo\frac{(n-\delta)^+}{n}\ \ \mbox{weak$^*$ in $L^\infty(\ot)$.}
	\end{equation*}
	We obtain
	\begin{equation*}
	\uo\frac{(n-\delta)^+}{n}=\eta_1(n-\delta)^+ \ \ \mbox{for each $\delta>0$}.
	\end{equation*}
	This implies that
	\begin{equation*}
	\uo=n\eta_1\ \ \mbox{on the set $\{n>0\}$.}
	\end{equation*}
	If $n=0$, then $\uo=0$, and we still have $\uo=n\eta_1$. This completes the proof of \eqref{ff}.
	Similarly, we can show
	\begin{equation*}
	\ute\nabla\weg^{\gamma}\rightarrow \ut\nabla n^\gamma\ \ \mbox{weakly in $\left(L^2(\ot)\right)^N$.}
	\end{equation*}
	We are ready to pass to the limit in \eqref{eeuo} and \eqref{eeut}, thereby finishing the proof of Theorem \ref{mth}.
	% to obtain the existence assertion in the case where $\mu=\nu$.
\end{proof}
\section{The limit as $\gamma\rightarrow\infty$ and proof of theorem \ref{mth1}}
Once again, the proof will be divided into several lemmas.  Now the solution to our problem \eqref{euo}-\eqref{uicon} is denoted by $(\nng,\nog,\ntg,\dg)$. That is, we have
\begin{eqnarray}
\pt\nng-\frac{\gamma}{\gamma+1}\Delta \left(\nng\right)^{\gamma+1}&=&G(\dg)\nng-D\ntg\equiv R^{(\gamma)}\ \ \mbox{in $\ot$,}\label{gen}\\
\partial_t\nog-\mdiv\left(\nog\nabla \left(\nng\right)^{\gamma}\right)&=& G(\dg)\nog - K_1(\dg)\nog \nonumber\\
&&+ K_2(\dg)\ntg\equiv R_1^{(\gamma)}\ \ \mbox{in $\ot$,}\label{geuo}\\
\partial_t\ntg-\mdiv\left(\ntg\nabla \left(\nng\right)^{\gamma}\right)&=& (G(\dg)-D)\ntg +K_1(\dg)\nog \nonumber\\
&&- K_2(\dg)\ntg\equiv R_2^{(\gamma)}\  \ \mbox{in $\ot$,}\label{geut}\\
b \pt \dg-\Delta \dg&=&-\psi(\dg)\nng+a\ntg\ \ \mbox{in $\ot$,}\label{ged}\\
\nog\nabla  \left(\nng\right)^{\gamma}\cdot\mathbf{n}&=&\ntg\nabla\left(\nng\right)^{\gamma} \cdot\mathbf{n}=0\  \ \mbox{on $\Sigma_T\equiv\po\times(0,T) $,}\label{guob}\\
%\nt\nabla p\cdot\mathbf{n}&=&0\  \ \mbox{on $\Sigma_T$,}\label{utb}\\
\dg&=&d_b\  \ \mbox{on $\Sigma_T$,}\label{gcb}\\
%	\nabla c\cdot\mathbf{n}&=&0\  \ \mbox{on $\Sigma_T$,}\label{cb}\\
\left.	\left(\nng,\nog,\ntg, \dg\right)\right|_{t=0}&=&\left(n_0(x)+\frac{1}{\gamma}, n_{01}(x)+\frac{1}{\gamma}, n_{02}(x), d_0(x)\right)\ \ \mbox{on $\Omega$. }\label{guicon}
\end{eqnarray}
As before, the term $\frac{1}{\gamma}$ is added in \eqref{guicon} to ensure that $\nng$ stays away from $0$ below. Therefore, it possesses enough regularity properties. We wish to find and identify the limit of solutions as $\gamma\rightarrow\infty$.
%We immediately have
By our analysis in the preceding section, we have
\begin{eqnarray}
\nog\geq 0,\ \ntg&\geq& 0,\ \nng=\nog+\ntg\leq c,\label{ngb1}\\
0\leq\dg&\leq& L,\label{ngb}
\end{eqnarray}
where $L$ is given as in \eqref{cl}. In \eqref{ngb1} and what follows, the generic positive number $c$ is independent of $\gamma$.
% , \ \frac{1}{\gamma}e^{-M_0T}\leq \nng\leq e^{2M_0T}\left(\|n_0\|_{\infty,\Omega}+\frac{1}{\gamma}\right),.
%where \begin{equation}\label{ng}
%	N_\gamma=\left\| n_0(x)+\frac{1}{\gamma}\right\|_{\infty,\Omega}.
%\end{equation}
We may assume that there is a subsequence of $(\nog,\ntg,\nng, \dg)$, not relabeled, such that
\begin{equation}
\nog\rightarrow\no^{(\infty)},\ \ \ntg\rightarrow\nt^{(\infty)},\ \ \nng\rightarrow n^{(\infty)},\ \ \dg\rightarrow d^{(\infty)}\ \ \mbox{weak$^*$ in $L^\infty(\ot)$.}
\end{equation}
\begin{lem}
	Assume that 
	\begin{equation}
	\pt d_b\in L^2(0,T; W^{1,2}(\Omega)),\ \  d_0\in W^{1,2}(\Omega).
	\end{equation}
Then we have
\begin{equation}\label{hope5}
\ioT\left(\pt\dg\right)^2dxdt	\leq c.
\end{equation}
Furthermore, if (H6) and (H8) hold, then we have
\begin{equation}\label{hope11}
	\|\nabla\dg\|_{\infty,\ot}\leq c.
\end{equation}
\end{lem}
\begin{proof} Use $\pt(\dg-d_b)$ as a test function in \eqref{ged} to get
	\begin{eqnarray}
\lefteqn{	b	\io\left(\pt\dg\right)^2dx+\frac{1}{2}\frac{d}{dt}\io|\nabla\dg|^2dx}\nonumber\\
&=&b	\io\pt\dg\pt d_bdx+\io\nabla\dg\cdot\nabla\pt d_bdx\nonumber\\
&&+\io\left(-\psi(\dg)\nng+a\ntg\right)\pt(\dg-d_b)dx.\nonumber
	\end{eqnarray}
%Observe that
%\begin{eqnarray}
%	\io\nabla\dg\cdot\nabla\pt d_bdx&=&\io\nabla(\dg-d_b)\cdot\nabla\pt d_bdx+\frac{1}{2}\frac{d}{dt}\io|\nabla d_b|^2dx
%\end{eqnarray}
Integrate to derive
\begin{equation}
	\ioT\left(\pt\dg\right)^2dxdt+\sup_{0\leq t\leq T}\io|\nabla\dg|^2dx\leq c.
\end{equation}	
With the aid of our assumptions (H6) and (H8), we can easily modify the proof of Proposition 2.3 in \cite{X3} to obtain \eqref{hope11}. The basic strategy there is to derive an equation for $\partial_{x_i}\dg$ and then apply a parabolic version of DeGiorgi iteration technique to the resulting equations. The boundary estimate is achieved by flattening the relevant portion of the boundary. All these steps can be carried out here. We shall omit the details.
% can be inferred from .put us in a position to invoke. Upon doing so, we
	The proof is complete.
	\end{proof}
Clearly, this lemma implies \eqref{dstro}.
%Now we may assume
%It is not difficult to see that 
%\begin{equation}\label{hope12}
%\dg\rightarrow d^{(\infty)}\ \ \mbox{strongly in $L^2(\ot)$ and weak$^*$ in $L^\infty(0,T; W^{1,\infty}(\Omega))$.}
%\end{equation}
Consequently,
\begin{equation}\label{sy10}
R^{(\gamma)}\rightarrow R^{(\infty)}=G(d^{(\infty)})n^{(\infty)}-D\nt^{(\infty)}\ \ \mbox{weak$^*$ in $L^\infty(\ot)$.}
\end{equation}

The core of our development is the following lemma.
\begin{lem}\label{l32}
	We have
	\begin{equation}\label{mass1}
		\ioT t\left(\vgm\right)^2dxdt+\ioT t\left|\nabla\vgm\right|^2dxdt\leq c.
	\end{equation}
\end{lem}
\begin{proof}Let $G_0$ be given as in Theorem \ref{mth1}. Then
	\begin{equation}\label{rbd}
		R^{(\gamma)}\leq G_0\nng.
	\end{equation}
	Use this in \eqref{gen} and multiply through the resulting inequality by $e^{-G_0t}$ to get
	\begin{equation}\label{we1}
		\pt\wg-\frac{\gamma e^{\gamma G_0t}}{\gamma+1}\Delta\left(\wg\right)^{\gamma+1}\leq 0\ \ \mbox{in $\ot$,}
	\end{equation}
	where
	$$\wg=e^{-G_0t}\nng.$$
	For each $\ve>0$ we let
	\begin{equation}
		\eta_\ve(s)=\left\{\begin{array}{ll}
			1 &\mbox{if $s>\ve$,}\\
			\frac{1}{\ve}s&\mbox{if $0\leq s\leq\ve$,}\\
			0&\mbox{if $s<0$.}
		\end{array}\right.\nonumber
	\end{equation}
	%Obviously, for each $s$ we have
	We can easily check that
	\begin{equation}
		\eta_\ve(s)\rightarrow \mbox{sgn}_0^+(s)=\left\{\begin{array}{ll}
			1 &\mbox{if $s>0$,}\\
			%\frac{1}{\ve}s&\mbox{if $0\leq s\leq\ve$,}\\
			0&\mbox{if $s\leq0$.}
		\end{array}\right. \ \ \mbox{as $\ve\rightarrow 0$.} \nonumber
	\end{equation}
	%For $\sigma\in(0,1)$ 
	%Fix
	%	$$0<\sigma<e^{-G_0T}.$$
	Let $\sigma\in \left(0, e^{-G_0T}\right)$ be given as in (H7). Clearly, $\eta_\ve\left(\wg-\sigma\right)\geq 0$. Multiply through \eqref{we1} by this function  to get
	\begin{eqnarray}
		\io\int_{0}^{\wg(x,t)}\eta_\ve\left(s-\sigma\right)dsdx
		&\leq& \io\int_{0}^{\wg(x,0)}\eta_\ve
		\left(s-\sigma\right)dsdx.\label{th10}
	\end{eqnarray}
	Take $\ve\ra 0$ in the above inequality to obtain
	\begin{eqnarray}
		\io \left(\wg(x,t)-\sigma\right)^+dx&\leq& \io \left(\wg(x,0)-\sigma\right)^+dx\nonumber\\
		&\leq&\left(\|n_0\|_{\infty,\Omega}+\frac{1}{\gamma}-\sigma\right)\left|\left\{n_0(x)+\frac{1}{\gamma}\geq \sigma\right\}\right|.\nonumber
	\end{eqnarray}
	Or equivalently,
	\begin{eqnarray}
		\io \left(\nng(x,t)-\sigma e^{G_0t}\right)^+dx
		%&\leq& \io \left(\wg(x,0)-\sigma\right)^+dx\nonumber\\
		&\leq& e^{G_0t}\left(\|n_0\|_{\infty,\Omega}+\frac{1}{\gamma}-\sigma\right)\left|\left\{n_0(x)+\frac{1}{\gamma}\geq \sigma\right\}\right|.\label{th12}
	\end{eqnarray}
	On the other hand,
	\begin{eqnarray}
		\io \left(\nng(x,t)-\sigma e^{G_0t}\right)^+dx	&\geq&\int_{\{\nng(x,t)\geq 1\}} \left(\nng(x,t)-\sigma e^{G_0t}\right)^+dx\nonumber\\
		&\geq& (1-\sigma e^{G_0t})\left|\left\{\nng(x,t)\geq 1\right\}\right|.\nonumber
	\end{eqnarray}
	This combined with \eqref{th12} implies
	\begin{eqnarray}
		\left|\left\{\nng(x,t)\geq 1\right\}\right|&\leq&\frac{e^{G_0t}\left(\|n_0\|_{\infty,\Omega}+\frac{1}{\gamma}-\sigma\right)}{1-\sigma e^{G_0t}} \left|\left\{n_0(x)+\frac{1}{\gamma}\geq \sigma\right\}\right|\nonumber\\
		&\ra&\frac{e^{G_0t}\left(\|n_0\|_{\infty,\Omega}-\sigma\right)}{1-\sigma e^{G_0t}} \left|\left\{n_0(x)\geq \sigma\right\}\right|\ \ (\mbox{as $\gamma\ra\infty$})\nonumber\\
		&\leq&\frac{e^{G_0t}\left(\|n_0\|_{\infty,\Omega}-\sigma\right)}{1-\sigma e^{G_0t}}\frac{1}{e^{G_0T}\|n_0\|_{\infty,\Omega}}|\Omega|.\label{th20}
	\end{eqnarray}
	The last step is due to our assumption (H7). We easily check
	$$\frac{e^{G_0t}\left(\|n_0\|_{\infty,\Omega}-\sigma\right)}{1-\sigma e^{G_0t}}<e^{G_0t}\|n_0\|_{\infty,\Omega}.$$
	Hence we can pick a number $\sigma_0\in\left(\frac{e^{G_0t}\left(\|n_0\|_{\infty,\Omega}-\sigma\right)}{1-\sigma e^{G_0t}}\frac{1}{e^{G_0T}\|n_0\|_{\infty,\Omega}},1 \right)$. Consequently,
	\begin{equation}\label{th21}
		\sup_{0\leq t\leq T}\left|\left\{\nng(x,t)\geq 1\right\}\right|\leq \sigma_0|\Omega|\ \ \mbox{at least for $\gamma$ sufficiently large}.
	\end{equation}
	Using $\left(\wg-\|n_0\|_{\infty,\Omega}-\frac{1}{\gamma}\right)^+$ as a test function in \eqref{we1}, we derive
	the weak maximum principle 
	\begin{equation}\label{th1}
		\wg\leq \|n_0\|_{\infty,\Omega}+\frac{1}{\gamma}\ \ \mbox{in $\ot$}.	
	\end{equation}
	This together with \eqref{rbd} implies
	\begin{equation}
		R^{(\gamma)}\leq G_0e^{G_0t}\left(\|n_0\|_{\infty,\Omega}+\frac{1}{\gamma}\right).
	\end{equation}
	%$$$$
	Let $\vgm$ be given as in \eqref{hope12}.
%	$$\vgm=\left(\nng\right)^{\gamma+1}.$$
	Use $t\vgm$ as a test function in \eqref{gen} to deduce
	\begin{eqnarray}
		\lefteqn{\frac{1}{\gamma+2}	\frac{d}{dt}\io t\left(\nng\right)^{\gamma+2}dx+\frac{\gamma t}{\gamma+1}\io\left|\nabla\vgm\right|^2dx}\nonumber\\
		&	=&\frac{1}{\gamma+2}\io \left(\nng\right)^{\gamma+2}dx+t\io R^{(\gamma)}\vgm dx\nonumber\\
		&\leq&\frac{e^{G_0T}\left(\|n_0\|_{\infty,\Omega}+\frac{1}{\gamma}\right)}{\gamma+2}\io \vgm dx+G_0e^{G_0T}\left(\|n_0\|_{\infty,\Omega}+\frac{1}{\gamma}\right)t\io \vgm dx.\label{hope1}
		%\nonumber\\
	%	&\leq& c\io \vgm dx.
	\end{eqnarray}
	Since
	$$	\left|\left\{\nng(x,t)\geq 1\right\}\right|+	\left|\left\{\nng(x,t)< 1\right\}\right|=|\Omega|,$$
	the inequality \eqref{th21} implies %By \eqref{uz}, we have
	$$\left|\left\{\nng(x,t)< 1\right\}\right|>(1-\sigma_0)|\Omega|.$$
	%We can easily check
	Evidently,
	$$\left(\vgm-1\right)^+= 0\ \ \mbox{on $\left\{\nng(x,t)< 1\right\}$.}$$
	This puts us in a position to apply Lemma \ref{poin}. Upon doing so, we arrive at
	\begin{equation}\label{sy13}
		\io \left(\vgm-1\right)^+dx\leq c\io \left|\nabla\left(\vgm-1\right)^+\right|dx=c\int_{\{\nng(x,t)\geq 1\}} \left|\nabla\vgm \right|dx.	
	\end{equation}
To estimate the first term on the right-hand side of \eqref{hope1}, we use $(\nng-1)^+$ as a test function in \eqref{gen} to get
\begin{equation}\label{hope2}
\sup_{0\leq t\leq T}	\io\left[(\nng-1)^+\right]^2dx+\gamma\ioT\left(\nng\right)^\gamma\left|\nabla(\nng-1)^+\right|^2dxdt\leq c.
\end{equation}
	%$$$$\frac{1}{2}\frac{d}{dt}
	For each $\varepsilon>0$ we estimate
	\begin{eqnarray}
		\io\vgm dx&=&\int_{\{\nng(x,t)\geq 1\}}\vgm dx+\int_{\{\nng(x,t)< 1\}}\vgm dx\nonumber\\
		&\leq&\io\left(\vgm-1\right)^+dx+c\nonumber\\
		&\leq &c\int_{\{\nng(x,t)\geq 1\}} \left|\nabla\vgm \right|dx+c\nonumber\\
		&=&c(\gamma+1)\io\left(\nng\right)^\gamma\left|\nabla(\nng-1)^+\right|dx+c\nonumber\\
		&\leq&\varepsilon \io \left(\nng\right)^\gamma dx+c(\varepsilon)(\gamma+1)^2\io\left(\nng\right)^\gamma\left|\nabla(\nng-1)^+\right|^2dx+c\nonumber\\
		&\leq&\frac{\varepsilon}{\|\nng\|_{\infty,\ot}} \io \vgm dx+c(\varepsilon)(\gamma+1)^2\io\left(\nng\right)^\gamma\left|\nabla(\nng-1)^+\right|^2dx+c.\nonumber
		% \left[\left(\vgm-1\right)^+\right]^2dx+\frac{c(\varepsilon)}{t}+c\nonumber\\
	%	&\leq&c\varepsilon t\io \left|\nabla\left(\vgm-1\right)^+\right|^2dx+\frac{c(\varepsilon)}{t}+c.
	\end{eqnarray}
By choose $\varepsilon$ suitably small, we immediately get
\begin{equation}
		\io\vgm dx\leq c(\gamma+1)^2\io\left(\nng\right)^\gamma\left|\nabla(\nng-1)^+\right|^2dx+c.
\end{equation}
	Use this in \eqref{hope1},  then integrate, and apply \eqref{hope2} to obtain
	\begin{eqnarray}
	\lefteqn{	\frac{1}{\gamma+2}	\sup_{0\leq t\leq T}\io t\left(\nng\right)^{\gamma+2}dx+\frac{\gamma}{\gamma+1}\ioT t\left|\nabla\vgm\right|^2dxdt}\nonumber\\
	&\leq &c(\gamma+1)\ioT\left(\nng\right)^\gamma\left|\nabla(\nng-1)^+\right|^2dxdt+c\ioT t\vgm dxdt+c\nonumber\\
	&\leq &c\ioT t|\nabla \vgm|dxdt+c\nonumber\\
	&\leq &\frac{\gamma}{2(\gamma+1)}\ioT t\left|\nabla\vgm\right|^2dxdt+c.\nonumber
	\end{eqnarray}
Consequently,
\begin{equation}
	\ioT t\left|\nabla\vgm\right|^2dxdt\leq c.\nonumber
\end{equation}
By a calculation similar to \eqref{sy13},
\begin{equation}
		\ioT t\left(\vgm\right)^2dxdt\leq 	\ioT t\left[\left(\vgm-1\right)^+\right]^2dxdt+c\leq c	\ioT t\left|\nabla\left(\vgm-1\right)^+\right|^2dxdt+c\leq c.\nonumber
	\end{equation}
	This completes the proof of Lemma \ref{l32}.
	\end{proof}

	We see that the sequence $\{\vgm\}$ is bounded in $L^2(\tau,T; W^{1,2}(\Omega))$ for each $\tau\in (0,T)$. Thus we may assume that \eqref{vweak} holds.
%	\begin{equation}\label{hope3}
	%	\vgm\ra v^{(\infty)}\ \ \mbox{weakly in $L^2(\tau,T; W^{1,2}(\Omega))$ for each $\tau\in (0,T)$. }
%	\end{equation}

%\begin{lem}

%\end{lem}
\begin{proof}[Proof of \eqref{sy1} and \eqref{sy4}] We shall employ an argument from \cite{FH}. For each $\delta>0$ define
	\begin{equation}
		\Omega^{(\gamma)}_\delta=	\left\{(x,t)\in \ot:\nng(x,t)\geq1+\delta\right\}.
	\end{equation}
	We argue by contradiction. Suppose that \eqref{sy1} is not true. Then there is a $\delta>0$ such that
	\begin{equation}\label{sy3}
		\left|\Omega^{(\infty)}_{2\delta}\right|>0.
		%\left\{(x,t)\in \ot:n^{(\infty)}(x,t)\geq1+2\delta\right\}\right
	\end{equation}
	We claim 
	\begin{equation}\label{sy2}
		\liminf_{\gamma\ra\infty}	\left|\Omega^{(\gamma)}_\delta\right|\equiv c_0>0.
	\end{equation}
	To see this, we estimate from \eqref{ngb} that
	\begin{eqnarray}
		\ioT\nng n^{(\infty)}\chi_{\Omega^{(\infty)}_{2\delta}}dxdt&=&\int_{\Omega^{(\infty)}_{2\delta}\cap\Omega^{(\gamma)}_\delta}\nng n^{(\infty)}dxdt+\int_{\Omega^{(\infty)}_{2\delta}\setminus\Omega^{(\gamma)}_\delta}\nng n^{(\infty)}dxdt\nonumber\\
		&\leq &e^{2G_0T}\left(\|n_0\|_{\infty,\Omega}+\frac{1}{\gamma}\right)^2	\left|\Omega^{(\gamma)}_\delta\right|+(1+\delta)\int_{\Omega^{(\infty)}_{2\delta}} n^{(\infty)}dxdt.\nonumber
	\end{eqnarray}	
	If $c_0$ in \eqref{sy2} is $0$, we take $\gamma\ra\infty$ in the above inequality to derive
	\begin{equation}
		\int_{\Omega^{(\infty)}_{2\delta}}n^{(\infty)}	n^{(\infty)}dxdt\leq (1+\delta)\int_{\Omega^{(\infty)}_{2\delta}} n^{(\infty)}dxdt.
	\end{equation}
	This is possible only if $\left|\Omega^{(\infty)}_{2\delta}\right|=0$. But this contradicts \eqref{sy3}. Thus \eqref{sy2} holds.
	On the other hand, for each $\tau\in(0,T)$ we have
	\begin{equation}
		c\geq \int_{\Omega^{(\gamma)}_\delta}t\vgm dxdt\geq\int_{\Omega^{(\gamma)}_\delta\cap(\Omega\times(\tau,T))}t\vgm dxdt\geq\tau (1+\delta)^{\gamma+1}\left|\Omega^{(\gamma)}_\delta\cap(\Omega\times(\tau,T))\right|.
	\end{equation}
That is,
$$\limsup_{\gamma\ra\infty}\left|\Omega^{(\gamma)}_\delta\cap(\Omega\times(\tau,T))\right|\leq 0\ \ \mbox{for  each $\tau\in(0,T)$. }$$
Obviously, this contradicts \eqref{sy2}. This completes the proof of \eqref{sy1}.
	%Let $\sigma\in(0,1)$ be given as in the lemma. Use $\left(\frac{1}{\sigma}-\frac{1}{\nng}\right)^+$ as a test function in \eqref{gen} to get\cap(\Omega\times(\tau,T))
	%	\begin{equation}
	%	\io \int_{0}^{\nng}\left(\frac{1}{\sigma}-\frac{1}{s}\right)^+dsdx\leq \io \int_{0}^{n_0+\frac{1}{\gamma}}\left(\frac{1}{\sigma}-\frac{1}{s}\right)^+dsdx+\ioT R^{(\gamma)}\left(\frac{1}{\sigma}-\frac{1}{\nng}\right)^+dxdt.
	%	\end{equation}
	
	Fix $\tau\in (0,T)$. First, we claim
	\begin{equation}\label{sy7}
		\lim_{\gamma\ra\infty}\int_{\tau}^{T}\io\left|1-\nng\right|\vgm dxdt=0.
	\end{equation}
To see this, let $\vep\in (0,1)$ be given. We estimate from \eqref{mass1} that
\begin{eqnarray}
	\int_{\tau}^{T}\io\left|1-\nng\right|\vgm dxdt&=&\int_{\{\left|1-\nng\right|\leq \vep\}\cap(\Omega\times(\tau,T))}\left|1-\nng\right|\vgm dxdt\nonumber\\
	&&+\int_{\{\nng>1+ \vep\}\cap(\Omega\times(\tau,T))}\left|1-\nng\right|\vgm dxdt\nonumber\\
	&&+\int_{\{\nng<1- \vep\}\cap(\Omega\times(\tau,T))}\left|1-\nng\right|\vgm dxdt\nonumber\\
	&\leq&c\vep+c\left|\{\nng>1+ \vep\}\cap(\Omega\times(\tau,T))\right|^{\frac{1}{2}}+c(1-\vep)^{\gamma+1}.\nonumber
\end{eqnarray}
Consequently,
\begin{equation}
	\limsup_{\gamma\ra\infty}\int_{\tau}^{T}\io\left|1-\nng\right|\vgm dxdt\leq c\vep.
\end{equation}
Since $\vep$ can be arbitrarily small, we yield \eqref{sy7}.

	Observe from \eqref{gen} that
%	$$\pt\wg=e^{-G_0t}\pt\nng-G_0e^{-G_0t}\nng=\frac{\gamma e^{-G_0t}}{\gamma+1}\Delta\vgm+e^{-G_0t}R^{(\gamma)}-G_0e^{-G_0t}\nng.$$
	 the sequence $\{\pt\nng\}$ is bounded in $L^2\left(\tau,T; \left(W^{1,2}(\Omega)\right)^*\right)$. We can infer from Lions-Aubin's lemma that $\{\nng\}$ is precompact in $C\left([\tau,T]; \left(W^{1,2}(\Omega)\right)^*\right)$. We may assume
	that
	\begin{equation}
		\nng\ra n^{(\infty)}\ \ \mbox{strongly in $C\left([\tau,T]; \left(W^{1,2}(\Omega)\right)^*\right)$.}
	\end{equation}
Once again, we pass to a subsequence if need be.
With this in mind, we can deduce from \eqref{vweak} that
\begin{eqnarray}
	\int_{\tau}^{T}\io\left(1-\nng\right)\vgm dxdt&=&\int_{\tau}^{T}\langle 1-\nng, \vgm\rangle dt\nonumber\\
	&=&\int_{\tau}^{T}\langle 1-n^{(\infty)}, v^{(\infty)}\rangle dt=\int_{\tau}^{T}\io\left( 1-n^{(\infty)}\right) v^{(\infty)}dxdt.\nonumber
\end{eqnarray}
This together with \eqref{sy7} and \eqref{sy1} implies
\begin{equation}
	\left( 1-n^{(\infty)}\right) v^{(\infty)}=0,
\end{equation}
from which \eqref{sy4} follows.\end{proof}

Now we are ready to prove \eqref{vstro}.
\begin{proof}[Proof of \eqref{vstro}]
	%The proof of \eqref{vstro} is similar to Lemma \ref{l33}. 
	Use $t^2\pt\vgm$ as a test function in \eqref{gen} to get
\begin{eqnarray}
	\lefteqn{(\gamma+1)t^2\io\left(\nng\right)^\gamma\left(\pt\nng\right)^2dx+\frac{\gamma}{2(\gamma+1)}\frac{d}{dt}\io t^2|\nabla\vgm|^2dx}\nonumber\\
	&=&\frac{\gamma}{\gamma+1}\io t |\nabla\vgm|^2dx+t^2\io R^{(\gamma)}\pt\vgm dx.\label{hope7}
\end{eqnarray}
To estimate the last integral in the above equation, we compute from \eqref{geut} that
\begin{eqnarray}
	-Dt^2\io\ntg\pt\vgm dx&=&-D\frac{d}{dt}\io t^2\ntg\vgm dx+2Dt\io \ntg\vgm dx+Dt^2\io\pt\ntg\vgm dx\nonumber\\
	&=&-D\frac{d}{dt}\io t^2\ntg\vgm dx+2Dt\io \ntg\vgm dx\nonumber\\
	&&-\frac{\gamma Dt^2}{\gamma+1}\io\frac{\ntg}{\nng}\left|\nabla\vgm\right|^2dx+Dt^2\io R_2^{(\gamma)}\vgm dx.\nonumber
\end{eqnarray}
Integrate and then apply \eqref{mass1} to deduce
\begin{equation}\label{hope6}
-D\int_{0}^{\tau}t^2\io\ntg\pt\vgm dxdt\leq c.
\end{equation}
Similarly,
%We compute from \eqref{} that
\begin{eqnarray}
	t^2\io G(\dg)\nng\pt\vgm dx&=&\frac{\gamma+1}{\gamma+2}\frac{d}{dt}\io t^2G(\dg)\left(\nng\right)^{\gamma+2}dx\nonumber\\
	&&-\frac{2(\gamma+1)t}{\gamma+2}\io G(\dg)\left(\nng\right)^{\gamma+2}dx\nonumber\\
	&&-\frac{\gamma+1}{\gamma+2}\io t^2G^\prime(\dg)\pt\dg\left(\nng\right)^{\gamma+2}dx.\nonumber
\end{eqnarray}
Integrate and then use (H5), \eqref{mass1} and \eqref{hope5} to derive
\begin{equation}\label{hope8}
	\int_{0}^{\tau}\io t^2G(\dg)\nng\pt\vgm dxdt\leq \frac{\gamma+1}{\gamma+2}\io \tau^2G(\dg)\nng\vgm dx+c.
\end{equation}
Integrate \eqref{hope7} and then take into consideration of \eqref{hope6} and \eqref{hope8} to obtain
\begin{eqnarray}
\lefteqn{	(\gamma+1)\int_{0}^{\tau}\io t^2\left(\nng\right)^\gamma\left(\pt\nng\right)^2dxdt}\nonumber\\
&&+\frac{\gamma}{2(\gamma+1)}\io \tau^2|\nabla\vgm|^2dx
\leq \frac{\gamma+1}{\gamma+2}\io \tau^2G(\dg)\nng\vgm dx+c.\label{hope9}
\end{eqnarray}
We easily infer from \eqref{sy13} that
\begin{eqnarray}
	\io\vgm dx\leq c\io|\nabla\vgm|dx+c\leq \varepsilon \io|\nabla\vgm|^2dx+c(\varepsilon),\ \ \varepsilon>0.\label{hope13}
\end{eqnarray}
Use this in \eqref{hope9} and choose $\varepsilon$ suitably small in the resulting inequality to derive
\begin{equation}\label{hope10}
	(\gamma+1)\ioT t^2\left(\nng\right)^\gamma\left(\pt\nng\right)^2dxdt+\sup_{0\leq t\leq T}\io t^2|\nabla\vgm|^2dx\leq c.	
\end{equation}
This combined with \eqref{hope13} yields
\begin{equation}\label{re1}
	\sup_{0\leq t\leq T}\io t^2\left(\vgm\right)^2dx\leq c.	
\end{equation}
Use $t^2\left(\vgm-v^{(\infty)}\right)$ as a test function in \eqref{gen} to deduce
\begin{eqnarray}
\lefteqn{\io t^2\pt\nng\left(\vgm-v^{(\infty)}\right)dx}\nonumber\\
&&+	\frac{t^2\gamma}{\gamma+1}\io\nabla\vgm\cdot\nabla\left(\vgm-v^{(\infty)}\right)dx=t^2\io R^{(\gamma)}\left(\vgm-v^{(\infty)}\right)dx.\label{hope4}
\end{eqnarray}
Note that 
\begin{equation}
	\io t^2\pt\nng\vgm dx=\frac{1}{\gamma+2}\frac{d}{dt}\io t^2\left(\nng\right)^{\gamma+2}dx-\frac{2t}{\gamma+2}\io \left(\nng\right)^{\gamma+2}dx.
\end{equation}
Integrate to get
\begin{eqnarray}
	\ioT t^2\pt\nng\vgm dxdt&=& \frac{1}{\gamma+2}\io T^2\left(\nng(x,T)\right)^{\gamma+2}dx-\frac{2}{\gamma+2}\ioT t \left(\nng\right)^{\gamma+2}dxdt\nonumber\\
	&\ra& 0\ \ \mbox{as $\gamma\ra\infty$}.\nonumber
\end{eqnarray}
The last step is due to \eqref{re1}.
Keeping this and \eqref{hope4} in mind, we calculate
\begin{eqnarray}
	\lefteqn{\limsup_{\gamma\ra\infty}\ioT  t^2\left|\nabla\left(\vgm-v^{(\infty)}\right)\right|^2dxdt}\nonumber\\	&\leq& \limsup_{\gamma\ra\infty}\ioT  t^2\nabla\vgm\cdot\nabla\left(\vgm-v^{(\infty)}\right)dxdt\nonumber\\
	&\leq&\int_{0}^{T}\langle t\pt n^{(\infty)}, tv^{(\infty)}\rangle dt+\limsup_{\gamma\ra\infty}\ioT  t^2R^{(\gamma)}\left(\vgm-v^{(\infty)}\right)dxdt.\label{hope14} %t\frac{\gamma}{\gamma+1}\int_{\tau}^{T}\io\nabla\vgm\cdot\nabla\left(\vgm-v^{(\infty)}\right)dxdt+\frac{\gamma}{\gamma+1}\int_{\tau}^{T}\io\nabla v^{(\infty)}\cdot\nabla\left(\vgm-v^{(\infty)}\right)dxdt\nonumber\\
%	&=&\int_{\tau}^{T}\io R^{(\gamma)}\left(\vgm-v^{(\infty)}\right)dxdt-\int_{\tau}^{T}\io\pt\nng\left(\vgm-v^{(\infty)}\right)dxdt\nonumber\\
%	&&+\frac{\gamma}{\gamma+1}\int_{\tau}^{T}\io\nabla v^{(\infty)}\cdot\nabla\left(\vgm-v^{(\infty)}\right)dxdt.
\end{eqnarray}
%\int_{\tau}^{T}
Observe that
\begin{equation}
	R^{(\gamma)}=\left(G(\dg)-G(d^{(\infty)})\right)\nng+G(d^{(\infty)})\nng-D\ntg.\nonumber
\end{equation}
Remember that $\{t\nng\},\{t\ntg\}$ are precompact in $C([0,T]; \left(W^{1,2}(\Omega)\right)^*)$. Furthermore, we have $G(d^{(\infty)})\in L^\infty(0,T; W^{1,\infty}(\Omega))$ due to (H5) and \eqref{hope11}. Hence
	\begin{eqnarray}
	\lefteqn{	\lim_{\gamma\ra\infty}\ioT  t^2R^{(\gamma)}\left(\vgm-v^{(\infty)}\right)dxdt}\nonumber\\
	&=&\lim_{\gamma\ra\infty}\int_{0}^{T}\left\langle  t\nng, tG(d^{(\infty)})\left(\vgm-v^{(\infty)}\right)\right\rangle dt\nonumber\\
	&&-D\lim_{\gamma\ra\infty}\int_{0}^{T}\left\langle  t\ntg, t\left(\vgm-v^{(\infty)}\right)\right\rangle dt=0.\label{re2}
	\end{eqnarray}
Use this in \eqref{hope14} to obtain
\begin{equation}\label{hope22}
	\limsup_{\gamma\ra\infty}\ioT  t^2\left|\nabla\left(\vgm-v^{(\infty)}\right)\right|^2dxdt\leq \int_{0}^{T}\langle t\pt n^{(\infty)}, tv^{(\infty)}\rangle dt.
\end{equation}
%To complete the proof of Theorem \ref{mth1}, we need to verify \eqref{com}. To this end, w
%We pass to the limit in \eqref{gen} to get
%\begin{equation}
%	\pt n^{(\infty)}-\Delta v^{(\infty)}=R^{(\infty)}\ \ \mbox{in $\ot$.}
%\end{equation}
We wish to show that the right-hand side of the above inequality is $0$. To this end, we introduce a function
%Let
$$\Psi^{(\infty)}(s)=\left\{\begin{array}{ll}
	0&\mbox{if $s\leq 1$,}\\
	\infty&\mbox{if $s>1$.}
\end{array}\right.$$
Obviously, $\Psi^{(\infty)}(s)$ is convex and lower semicontinuous (\cite{H}, p.49) and
\begin{equation}
	\partial\Psi^{(\infty)}(s)=\vp_{\infty}(s),
	\end{equation}where $\vp_{\infty}(s)$ is given as in \eqref{hope20}. We claim that $t\mapsto\io\Psi^{(\infty)}(\nnf(x,t))dx$ is an absolutely continuous function on $ (0,T)$ and
\begin{equation}\label{rrt4}
	\frac{d}{dt}\io\Psi^{(\infty)}(\nnf(x,t))dx=\langle \pt n^{(\infty)}, v^{(\infty)}\rangle\ \ \mbox{for a.e. $t\in(0,T)$}.
\end{equation}
%\left(W^{1,2}(\Omega)\right)^*
%Even though ,  we can easily derive from  that the conclusions of
Note that Lemma \ref{comt} is not applicable here because we do not have $\nnf\in L^2(\tau,T; W^{1,2}(\Omega))$. We shall give a direct proof.
%That is, $t\mapsto\io\Psi^{(\infty)}(\nnf(x,t))dx$ is an absolutely continuous function on $ (0,T)$ and
%We compute the subgradient $\partial\Psi^{(\infty)}$ of $\Psi^{(\infty)}(s)$ to get
To do this, we infer from \eqref{re3} that
%We can   that This combined with 
\begin{eqnarray}
	\lefteqn{\io\Psi^{(\infty)}(\nnf(x,t+\varepsilon))dx-\io\Psi^{(\infty)}(\nnf(x,t))dx}\nonumber\\
	&\geq&\io v^{(\infty)}(x,t)\left(\nnf(x,t+\varepsilon)-\nnf(x,t)\right)dx\nonumber\\
	&=&\langle\nnf(\cdot,t+\varepsilon)-\nnf(\cdot,t),v^{(\infty)}(\cdot,t) \rangle,\ \ \varepsilon>0.\label{rrt1}
\end{eqnarray}
Let $\zeta(t)\in C_0^\infty(0,T)$ be such that $\zeta(t)\geq 0$. Multiply through \eqref{rrt1} by $\frac{1}{\varepsilon}\zeta(t)$, integrate the resulting inequality over $(0,T)$, and thereby obtain
\begin{eqnarray}
\lefteqn{	\int_{0}^{T}\io\Psi^{(\infty)}(\nnf(x,t))dx\frac{\zeta(t-\varepsilon)-\zeta(t)}{\varepsilon}dt}\nonumber\\
&\geq&\int_{0}^{T}\left\langle\frac{\nnf(\cdot,t+\varepsilon)-\nnf(\cdot,t)}{\varepsilon},v^{(\infty)}(\cdot,t) \right\rangle\zeta(t)dt\ \ \mbox{ for $\varepsilon$ suitably small}.\label{rrt2}
\end{eqnarray}
We can easily take the limit $\varepsilon\ra 0$ on the left-hand side of the preceding inequality. To show that we can do the same for the right-hand side, we integrate \eqref{re4} over $(t, t+\varepsilon)$ to get
%, \eqref{hope10}, and \eqref{re1} that
\begin{equation}
\nnf(x,t+\varepsilon)-\nnf(x,t)-\Delta\int_{t}^{t+\varepsilon}v^{(\infty)}ds	=\int_{t}^{t+\varepsilon}R^{(\infty)}(x,s)ds\ \ \mbox{in $\Omega$.}
\end{equation}
Using $\frac{1}{\varepsilon}v^{(\infty)}(x,t)$ as a test function in the above equation gives
%Therefore,
\begin{eqnarray}
	\lefteqn{\left\langle\frac{\nnf(\cdot,t+\varepsilon)-\nnf(\cdot,t)}{\varepsilon},v^{(\infty)}(\cdot,t) \right\rangle}\nonumber\\
	&=&-\io\frac{1}{\varepsilon}\int_{t}^{t+\varepsilon}\nabla v^{(\infty)}ds\cdot\nabla v^{(\infty)}dx+\io\frac{1}{\varepsilon}\int_{t}^{t+\varepsilon}R^{(\infty)}(x,s)dsv^{(\infty)}dx.
	%\nonumber\\
%	&\ra&-\io\left|\nabla v^{(\infty)}\right|^2dx+\io R^{(\infty)}v^{(\infty)}dx=\left\langle\pt\nnf,v^{(\infty)}(\cdot,t)\right\rangle\ \ \mbox{as $\varepsilon\ra 0$.}
\end{eqnarray}
We can verify
\begin{equation}\label{rrt3}
	\frac{1}{\varepsilon}\int_{t}^{t+\varepsilon}\nabla v^{(\infty)}(x,s)ds\ra \nabla v^{(\infty)}(x,t)\ \ \mbox{a.e on $\ot$.}
\end{equation}
It follows from \eqref{hope10} that
\begin{equation}
\sup_{t\geq\tau}	\left\|	\frac{1}{\varepsilon}\int_{t}^{t+\varepsilon}\nabla v^{(\infty)}(x,s)ds\right\|_{2,\Omega}\leq c(\tau),\ \ \tau\in (0,T).
\end{equation}
This together with \eqref{rrt3} implies
\begin{equation}
		\frac{1}{\varepsilon}\int_{t}^{t+\varepsilon}\nabla v^{(\infty)}(x,s)ds\ra \nabla v^{(\infty)}(x,t)\ \ \mbox{weakly in $\left(L^2(\Omega\times(\tau,T))\right)^N$.}
\end{equation}
Similarly,
$$ 	\frac{1}{\varepsilon}\int_{t}^{t+\varepsilon} R^{(\infty)}(x,s)ds\ra R^{(\infty)}(x,t)\ \ \mbox{strongly in $L^q(\ot)$ for each $q>1$.}$$
We are ready to evaluate
\begin{eqnarray}
\lefteqn{	\lim_{\varepsilon\ra 0}\int_{0}^{T}\left\langle\frac{\nnf(\cdot,t+\varepsilon)-\nnf(\cdot,t)}{\varepsilon},v^{(\infty)}(\cdot,t) \right\rangle\zeta(t)dt}\nonumber\\
&=&-\int_{0}^{T}\io\left|\nabla v^{(\infty)}\right|^2dx\zeta(t)dt+\int_{0}^{T}\io R^{(\infty)}v^{(\infty)}dx\zeta(t)dt\nonumber\\
&=&\int_{0}^{T}\left\langle\pt\nnf,v^{(\infty)}(\cdot,t)\right\rangle\zeta(t)dt
\end{eqnarray}
The last step is due to \eqref{re4}. Taking $\varepsilon\ra 0$ in \eqref{rrt2} yields
$$\frac{d}{dt}\io\Psi^{(\infty)}(\nnf(x,t))dx\geq\langle \pt n^{(\infty)}, v^{(\infty)}\rangle\ \ \mbox{in the sense of distributions}.$$
Replacing each occurrence of $\varepsilon$ by $-\varepsilon$ in the preceding proof, we can derive
$$\frac{d}{dt}\io\Psi^{(\infty)}(\nnf(x,t))dx\leq\langle \pt n^{(\infty)}, v^{(\infty)}\rangle\ \ \mbox{in the sense of distributions}.$$ This completes the proof of \eqref{rrt4}
%Therefore,

With \eqref{rrt4} in mind, we calculate
\begin{eqnarray}
\int_{0}^{T}\langle t\pt n^{(\infty)}, tv^{(\infty)}\rangle dt&=&\int_{0}^{T}t^2\langle \pt n^{(\infty)}, v^{(\infty)}\rangle dt\nonumber\\
&=&	\int_{0}^{T}t^2 \frac{d}{dt}\io\Psi^{(\infty)}(\nnf(x,t))dxdt\nonumber\\
&=&\int_{0}^{T} \frac{d}{dt}\io t^2\Psi^{(\infty)}(\nnf(x,t))dxdt-2\int_{0}^{T} \io t\Psi^{(\infty)}(\nnf(x,t))dxdt\nonumber\\
&=&0.\label{hope21}
\end{eqnarray}
The last step is due to the fact that $\Psi^{(\infty)}(\nnf(x,t))\equiv0$. Combining \eqref{hope21} with \eqref{hope22} yields \eqref{vstro}.

To complete the proof of Theorem \ref{mth1}, we still need to verify \eqref{com}. To this end, we multiply through \eqref{gen} by $\vgm$ to get
$$\frac{1}{\gamma+2}\pt\left(\nng\right)^{\gamma+2}-\frac{\gamma}{\gamma+1}\left(\mdiv(\vgm\nabla\vgm)-|\nabla\vgm|^2\right)=R^{(\gamma)}\vgm.$$
Even though it is not clear if $\{t\vgm\}$ is precompact in $L^2(\ot)$ because we do not have any estimates on $\pt\vgm$, \eqref{vstro} and \eqref{re2} are enough to justify passing to the limit in the above equation, thereby obtaining \eqref{com}. This finishes the proof of Theorem \ref{mth1}.
\end{proof}

\end{document}